\documentclass[a4paper,english]{article}
\usepackage{babel}
\usepackage{zibtitlepage}

\usepackage[utf8]{inputenc}

\usepackage{geometry}

\usepackage{amsmath}
\usepackage{amsfonts}
\usepackage{amssymb}
\usepackage{mathtools}

\usepackage[ruled,linesnumbered,vlined]{algorithm2e}

\newcommand{\goto}{\textbf{go to}}
\SetCommentSty{mycommfont}
\SetKwInOut{Input}{in}
\SetKwInOut{Output}{out}
\SetKwRepeat{Do}{do}{while}

\usepackage{tabularx,supertabular,booktabs,array}
\usepackage{pdflscape}
\usepackage{longtable}

\usepackage{tikz}
\usepackage{etex}
\usepackage{pgfplots}
\usetikzlibrary{shapes.gates.logic.US,shapes.geometric}
\usetikzlibrary{shapes,arrows,backgrounds,topaths}
\usetikzlibrary{arrows,decorations.markings}
\usetikzlibrary{matrix,snakes}
\usetikzlibrary{plotmarks}
\usetikzlibrary{decorations.text}
\usetikzlibrary{patterns}

\usepackage{pgfplots}

\usepackage{url}
\usepackage[bookmarks,implicit=true,colorlinks,linkcolor=black,citecolor=black,urlcolor=blue]{hyperref}

\usepackage{xcolor}
\usepackage{xspace}
\usepackage{subcaption}
\usepackage{xargs}
\usepackage[colorinlistoftodos,prependcaption,textsize=footnotesize]{todonotes}
\usepackage{gensymb}
 \newcommand{\name}[1]{\mbox{#1}\xspace}

\newcommand{\scip}{\name{SCIP}}
\newcommand{\baron}{\name{BARON}}

\newcommand{\cplex}{\name{CPLEX}}
\newcommand{\cppad}{\name{CppAD}}
\newcommand{\ipopt}{\name{Ipopt}}
\newcommand{\mumps}{\name{MUMPS}}
\newcommand{\randset}{\name{\textsc{rand}}}
\newcommand{\realset}{\name{\textsc{real}}}

\newcommand{\MINLP}{\name{MINLP}}
\newcommand{\MINLPs}{\name{MINLPs}}
\newcommand{\MILP}{\name{MILP}}
\newcommand{\MILPs}{\name{MILPs}}
\newcommand{\NLP}{\name{NLP}}
\newcommand{\NLPs}{\name{NLPs}}

\newcommand{\IP}{\name{IP}}

\newcommand{\LP}{\name{LP}}
\newcommand{\LPs}{\name{LPs}}

\newcommand{\rcpp}{\name{RCPP}}

\newcommand{\cpp}{\name{CPP}}
\newcommand{\cpps}{\name{CPPs}}
\newcommand{\grasp}{\name{GRASP}}
\newcommand{\ringset}{\mathcal{R}}
\newcommand{\nrings}{n}
\newcommand{\ntypes}{T}
\newcommand{\typeset}{\mathcal{T}}
\newcommand{\type}{\tau}

\newcommand{\F}{\mathcal{F}}
\newcommand{\Fprime}{\mathcal{F'}}
\newcommand{\DW}{DW}
\newcommand{\PDW}{PDW}

\newcommand{\redcosts}{red}

\newcommand{\cp}{\mathcal{CP}}
\newcommand{\cpdom}{\mathcal{CP}^{*}}
\newcommand{\cpfeas}{\mathcal{CP}_{feas}}
\newcommand{\cpinfeas}{\mathcal{CP}_{infeas}}
\newcommand{\cpunknown}{\mathcal{CP}_{unknown}}

\newcommand{\rp}{\mathcal{RP}}

\newcommand{\rext}{R}
\newcommand{\rint}{r}
\newcommand{\demand}{D}
\newcommand{\width}{W}
\newcommand{\height}{H}
\newcommand{\circleset}{\mathcal{C}}
\newcommand{\mrp}{\mathcal{MP}}

\newcommand{\mvec}[2]{\begin{pmatrix}#1\\#2\end{pmatrix}}

\newcommand{\norm}[1]{\left\|#1\right\|}

\newcommand{\floor}[1]{\left\lfloor{#1}\right\rfloor}
\newcommand{\ceil}[1]{\left\lceil{#1}\right\rceil}
\newcommand{\R}{\mathbb{R}}
\newcommand{\Z}{\mathbb{Z}}

\newcommand{\dash}{---}
\newcommand{\fa}{\text{ for all }}

\definecolor{c1}{HTML}{01DF01}
\definecolor{c2}{HTML}{F7FE2E}
\definecolor{c3}{HTML}{DBA901}
\definecolor{c4}{HTML}{B45F04}
\definecolor{c5}{HTML}{DF3A01}

\newcommand{\bardiagram}[2]{
  \def\yscale{2.0 / 30.0}
  \begin{tikzpicture}[scale=1]

    \draw[->] (0,0) -- (0,55 * \yscale) node[above]{\# instances};
    \draw[->] (0,0) -- (6.0,0) node[right]{#1};

    \draw[] (0.0,0 * \yscale) -- (-0.2,0 * \yscale) node[left]{0\%};
    \draw[] (0.0,25 * \yscale) -- (-0.2,25 * \yscale) node[left]{25\%};
    \draw[] (0.0,50 * \yscale) -- (-0.2,50 * \yscale) node[left]{50\%};

    \foreach \idx\desc\a\b\c\d\e in {#2}
    {
      \def\startx{0.2 + \idx * 1.5};
      \def\rectwidth{0.2};
      \draw (\startx + 0.5, -0.1) -- (\startx + 0.5,-0.1) node[below]{$\desc$};

      \draw[fill=c1] (\startx,       0) rectangle (\startx       + \rectwidth,\a * \yscale);
      \draw[fill=c2] (\startx + 0.2, 0) rectangle (\startx + 0.2 + \rectwidth,\b * \yscale);
      \draw[fill=c3] (\startx + 0.4, 0) rectangle (\startx + 0.4 + \rectwidth,\c * \yscale);
      \draw[fill=c4] (\startx + 0.6, 0) rectangle (\startx + 0.6 + \rectwidth,\d * \yscale);
      \draw[fill=c5] (\startx + 0.8, 0) rectangle (\startx + 0.8 + \rectwidth,\e * \yscale);
    }

  \end{tikzpicture}
}

\newcommand{\tabledefline}[2]{\multicolumn{1}{l}{\rlap{#1\ \dash\ #2}}\\}

\definecolor{plum}{rgb}{0.8, 0.6, 0.8}
\definecolor{applegreen}{rgb}{0.55, 0.71, 0.0}
\definecolor{apricot}{rgb}{0.98, 0.81, 0.69}
\definecolor{amber}{rgb}{1.0, 0.49, 0.0}
\definecolor{americanrose}{rgb}{1.0, 0.01, 0.24}

\newcommandx{\change}[2][1=]{\todo[linecolor=americanrose,backgroundcolor=americanrose!25,bordercolor=americanrose,#1]{#2}\,}
\newcommandx{\improvement}[2][1=]{\todo[linecolor=amber,backgroundcolor=amber!25,bordercolor=amber,#1]{#2}\,}
\newcommandx{\unsure}[2][1=]{\todo[linecolor=apricot,backgroundcolor=apricot!25,bordercolor=apricot,#1]{#2}\,}
\newcommandx{\info}[2][1=]{\todo[linecolor=applegreen,backgroundcolor=applegreen!25,bordercolor=applegreen,#1]{#2}\,}
\newcommandx{\missing}[2][1=]{\todo[linecolor=blue,backgroundcolor=blue!25,bordercolor=blue,#1]{#2}\,}

\tikzstyle{chart}=[
    legend label/.style={font={\scriptsize},anchor=west,align=left},
    legend box/.style={rectangle, draw, minimum size=5pt},
    axis/.style={black,semithick,->},
    axis label/.style={anchor=east,font={\tiny}},
]

\tikzstyle{pie chart}=[
    chart,
    slice/.style={line cap=round, line join=round, draw=black},
    pie title/.style={font={\bf}},
    slice type/.style 2 args={
        ##1/.style={fill=##2},
        values of ##1/.style={}
    }
]

\pgfdeclarelayer{background}
\pgfdeclarelayer{foreground}
\pgfsetlayers{background,main,foreground}

\newcommand{\pie}[3][]{
    \begin{scope}[#1]
    \pgfmathsetmacro{\curA}{90}
    \pgfmathsetmacro{\r}{1}
    \def\c{(0,0)}
    \node[pie title] at (90:1.3) {#2};
    \foreach \v/\s in{#3}{
        \pgfmathsetmacro{\deltaA}{\v/100*360}
        \pgfmathsetmacro{\nextA}{\curA + \deltaA}
        \pgfmathsetmacro{\midA}{(\curA+\nextA)/2}

        \path[slice,\s] \c
            -- +(\curA:\r)
            arc (\curA:\nextA:\r)
            -- cycle;
        \pgfmathsetmacro{\d}{max((\deltaA * -(.5/50) + 1) , .5)}

        \begin{pgfonlayer}{foreground}
        \path \c -- node[pos=\d,pie values,values of \s]{$\v\%$} +(\midA:\r);
        \end{pgfonlayer}

        \global\let\curA\nextA
    }
    \end{scope}
}

\newcommand{\revision}[1]{#1}

\newcommand{\volcol}{vol}
\newcommand{\dualcol}{dual}
\newcommand{\primalcol}{primal}
\newcommand{\graspcol}{\grasp}
\newcommand{\cpfoundcol}{nfeas}
\newcommand{\cpcandcol}{ncands}
\newcommand{\rpcol}{nrp}
\newcommand{\nodescol}{nodes}
\newcommand{\enumtimecol}{enum time}

\usepackage{amsthm}
\newtheorem{definition}{Definition}
\newtheorem{theorem}{Theorem}
\newtheorem{example}{Example}
\newtheorem{remark}{Remark}

\begin{document}

\let\pdfoutorg\pdfoutput
\let\pdfoutput\undefined
\ZTPAuthor{Ambros Gleixner\\
           Stephen J.\ Maher\\
           Benjamin M\"uller\\
           Jo\~{a}o Pedro Pedroso
}
\ZTPTitle{\bf Exact Methods for Recursive Circle Packing}
\ZTPInfo{This work has been supported by the Research Campus MODAL
  \emph{Mathematical Optimization and Data Analysis Laboratories} funded by the
  Federal Ministry of Education and Research (BMBF Grant~05M14ZAM).  All
  responsibilty for the content of this publication is assumed by the authors.}
\ZTPNumber{17-07}
\ZTPMonth{January}
\ZTPYear{2018}
\ZTPInfo{This report has been published in Annals of Operations Research. Please cite as:\\[1ex] Gleixner et al., Price-and-verify: a new algorithm for recursive circle packing using Dantzig–Wolfe decomposition. Annals of Operations Research, 2018, \href{http://dx.doi.org/10.1007/s10479-018-3115-5}{\textcolor{blue}{DOI:10.1007/s10479-018-3115-5}}}

\title{\bf Exact Methods for Recursive Circle Packing}

\author{Ambros Gleixner\thanks{Zuse Institute Berlin, Takustr.~7, 14195~Berlin, Germany, \texttt{\{benjamin.mueller,gleixner,maher\}@zib.de}}
   \and Stephen J.\ Maher$^*$
   \and Benjamin M\"uller$^*$
   \and Jo\~{a}o Pedro Pedroso\thanks{Faculdade de Ciências, Universidade do Porto, Rua do Campo Alegre, 4169-007 Porto, Portugal, \texttt{jpp@fc.up.pt}}}

\zibtitlepage
\maketitle

\begin{abstract}
  Packing rings into a minimum number of rectangles is an optimization problem
which appears naturally in the logistics operations of the tube industry. It
encompasses two major difficulties, namely the positioning of rings in
rectangles and the recursive packing of rings into other rings. This problem is
known as the Recursive Circle Packing Problem (\rcpp).
We present the first \revision{dedicated method for solving \rcpp that provides strong dual bounds} based on an exact Dantzig-Wolfe
reformulation of a nonconvex mixed-integer nonlinear programming
formulation. The key idea of this reformulation is to break symmetry on each
recursion level by enumerating one-level packings, i.e., packings of circles
into other circles, and by dynamically generating packings of circles into
rectangles.
We use column generation techniques to design a ``price-and-verify'' algorithm
that solves this reformulation to global optimality.
Extensive computational experiments on a large test set show that
our method not only computes tight dual bounds, but often produces primal
solutions better than those computed by heuristics from the literature.
 \end{abstract}

\section{Introduction}\label{section:introduction}

Packing problems appear naturally in a wide range of real-world
applications. They contain two major difficulties. First, complex geometric
objects need to be selected, grouped, and packed into other objects, e.g.,
warehouses, containers, or parcels. Second, these objects need to be packed in a
dense, non-overlapping way and, depending on the application, respect some
physical constraints. Modeling packing problems mathematically often leads to
nonconvex mixed-integer nonlinear programs (\MINLPs). Solving them to global
optimality is a challenging task for state-of-the-art \MINLP solvers.
In this paper, we will study exact algorithms for solving a particularly complex
version involving the \emph{recursive} packing of 2-dimensional rings.
This variant has real-world applications in the tube industry.

The Recursive Circle Packing Problem (\rcpp) has been introduced recently by
Pedroso et al.~\cite{Pedroso2016}. The objective of \rcpp is to select a minimum
number of rectangles of the same size such that a given set of rings can be
packed into these rectangles in a non-overlapping way. A ring is characterized
by an internal and an external radius. Rings can be put recursively into larger
ones or directly into a rectangle. A set of rings packed into a rectangle is
called a \textit{feasible packing} if and only if all rings lie within the
boundary of the rectangle and do not intersect each
other. Figure~\ref{fig:prob:solu} gives two examples of a feasible packing.
\begin{figure}[h]
  \clearpage{}\centering
\begin{minipage}{0.30\textwidth}
  \centering
  \begin{tikzpicture}[scale=0.4]
    \draw[thick] (0,0) -- (8.0,0) -- (8.0,7.0) -- (0, 7.0) -- cycle;
    \foreach \id/\x/\y/\R/\r in
             {
               1/1.9/1.9/1.9/1.8,
               2/0.96/1.8/0.83/0.4,
               2/2.25/2.8/0.83/0.4,
               2/2.5/1.15/0.83/0.4,
               1/6.1/1.9/1.9/1.8,
               2/5.16/1.8/0.83/0.4,
               2/6.45/2.8/0.83/0.4,
               2/6.7/1.15/0.83/0.4,
               1/4.0/5.1/1.9/1.8,
               2/3.06/5.0/0.83/0.4,
               2/4.35/6.0/0.83/0.4,
               2/4.6/4.35/0.83/0.4,
               2/1.4/6.0/0.83/0.4,
               2/0.83/4.45/0.83/0.4,
               2/6.6/6.0/0.83/0.4,
               2/7.17/4.45/0.83/0.4
             }
             \draw[fill=black,even odd rule] (\x,\y) circle (\R) circle (\r);
  \end{tikzpicture}
\end{minipage}
\begin{minipage}{0.30\textwidth}
  \centering
  \begin{tikzpicture}[scale=0.4]
    \draw[thick] (0,0) -- (8.0,0) -- (8.0,7.0) -- (0, 7.0) -- cycle;
    \foreach \id/\x/\y/\R/\r in
             {
               2/6.8/5.8/1.2/1.0,
               3/2.00128/5.0/2.0/0.4,
               8/4.721279/6.098191/0.9/0.6,
               9/0.9/2.317016/0.9/0.6,
               12/5.551494/2.4/2.4/2.2,
               13/2.14/0.9/0.9/0.6,
               12/4.6/2.8/1.2/1.0,
               13/4.6/2.8/0.9/0.6,
               13/6.0/1.2/0.9/0.6,
               13/6.7/3.0/0.9/0.6
             }
             \draw[fill=black,even odd rule] (\x,\y) circle (\R) circle (\r);
  \end{tikzpicture}
\end{minipage}
\clearpage{}
  \caption{Two feasible packings of rings into rectangles.}
  \label{fig:prob:solu}
\end{figure}
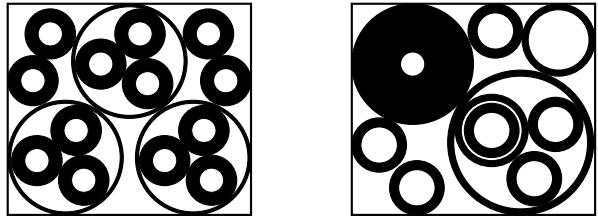

Pedroso et al.~\cite{Pedroso2016} present a nonconvex \MINLP formulation for
\rcpp. The key idea of the developed \MINLP is to use binary variables in order
to indicate whether a ring is packed inside another larger ring or directly into
a rectangle. Due to a large number of binary variables and nonconvex quadratic
constraints, the model is not of practical relevance and can only be used to
solve very small instances. However, the authors present a well-performing local
search heuristic to find good feasible solutions.

The purpose of this article is to present the first exact column generation
algorithm for \rcpp based on a Dantzig-Wolfe decomposition, for which, so far,
only heuristics exist. The new method is able to solve small- and medium-size
instances to global optimality and computes good feasible solutions
  and tight dual bounds for larger ones.
We first develop a reformulation, which is similar to the classical
reformulation for the Cutting Stock Problem~\cite{Gilmore1961}, however,
featuring nonlinear and nonconvex sub-problems. This formulation breaks the
symmetry between equivalent rectangles.
As a second step, we combine this reformulation with an enumeration scheme
for patterns that are characterized by rings packed inside other rings. Such
patterns only allow for a one-level recursion and break the symmetry between
rings with the same internal and external radius in each rectangle.
Finally, we present how to solve the resulting reformulation with a
column generation algorithm.

Finally, we present how to solve the resulting reformulation with a specialized
column generation algorithm that we call \emph{price-and-verify}. The
algorithm does not only generate new patterns but also verifies
patterns that are found during the enumeration scheme dynamically, i.e., only when required for
solving the continuous relaxation of the reformulation.

\section{Background}

The problem of finding dense packings of geometric objects has a rich history
that goes back to Kepler's Conjecture in 1611, which has been proven recently by
Hales et al.~\cite{Hales2015}. Packing identical or diverse objects into
different geometries like circles, rectangles, and polygons remains a relevant
topic and has been the focus of much research during the last decades.
The survey by Hifi and M'Hallah~\cite{Hifi2009} reviews the most relevant
results for packing 2- and 3-dimensional spheres into regions in the Euclidean
space. Applications, heuristics, and exact strategies of packing arbitrarily
sized circles into a container are presented by Castillo et
al.~\cite{Castillo2008}.
Costa et al.~\cite{Costa2013} use a spatial branch-and-bound algorithm to find
good feasible solutions for the case of packing identical circles. They propose
different symmetry-breaking constraints to tighten the convex relaxation and
improve the success rate of local nonlinear programming algorithms.

Closely related to packing problems are plate-cutting problems, in which convex
objects, e.g., circles, rectangles, or polygons, are cut from different convex
geometries~\cite{Dror1999}. Modeling these problems leads typically to nonlinear
programs, for which exact mathematical programming solutions are described by
Kallrath~\cite{Kallrath2009}.
Minimizing the volume of a box for an overlap-free placement of ellipsoids was
studied by Kallrath~\cite{Kallrath2015}. His closed, non-convex nonlinear
programming formulation uses the entries of the rotation matrix as variables and
can be used to get feasible solutions for instances with up to 100 ellipsoids.

The general type of the problem finds application in the tube industry, where
shipping costs represent a large fraction of the total costs of product
delivery. Tubes will be cut to the same length of the container in which they
may be shipped. In order to reduce idle space inside containers, smaller tubes
might be placed inside larger ones. Packing tubes as densely as possible reduces
the total number of containers needed to deliver all tubes and thus has a large
impact on the total cost.

\rcpp is a generalization of the well-known,
$\mathcal{NP}$-hard~\cite{Lenstra1979,Demaine2010} Circle Packing Problem (\cpp) and
therefore is also $\mathcal{NP}$-hard. Reducing an instance of \cpp to \rcpp can
be done by setting all internal radii to zero, i.e., by forbidding to pack rings
into larger rings. Typically, problems like \rcpp contain multiple sources of
symmetry. Any permutation of rectangles, i.e., relabeling rectangles,
constitutes an equivalent solution to \rcpp. Even worse, there is also
considerable symmetry inside a rectangle. First, rotating or reflecting a
rectangle gives an equivalent rectangle since both contain the same set of
rings. Second, two rings with same internal and external radius can be exchanged
arbitrarily inside a rectangle, again resulting in an equivalent rectangle
packing.

One possible way to break the symmetry of \rcpp is to add symmetry-breaking
constraints, which have been frequently used for scheduling problems, e.g.,
lot-sizing problems~\cite{Jans2009}.
An alternative approach is the use of decomposition techniques. These techniques
aggregate identical sub-problems in order to reduce symmetry and typically
strengthen the relaxation of the initial formulation. A well-known decomposition
technique is the Dantzig-Wolfe decomposition~\cite{Dantzig1960}. It is an
algorithm to decompose Linear Programs (\LPs) into a master and, in general,
several sub-problems. Column generation is used with this decomposition to
improve the solvability of large-scale linear programs. Embedded in a
branch-and-bound algorithm, it can be used to solve mixed-integer linear
programs (\MILPs), e.g., bin packing~\cite{Vanderbeck2010}, two dimensional bin
packing~\cite{Pisinger2005}, cutting stock~\cite{Gilmore1961,Gilmore1965},
multi-stage cutting stock~\cite{Muter2013}, and many other problems.
Fortz et al.~\cite{Fortz2010} applied a Dantzig-Wolfe decomposition to a
stochastic network design problem with a convex nonlinear objective function.
While most implementations of Dantzig-Wolfe are problem-specific, a number of
frameworks have been implemented for automatically applying a Dantzig-Wolfe
decomposition to a compact \MILP formulation,
see~\cite{Bergner2015,Gamrath2010,Ralphs2005}.

In this paper we apply a Dantzig-Wolfe decomposition to an \MINLP formulation of
\rcpp and present the first exact solution method that is able to solve
practically relevant instances.
The rest of the paper is organized as follows. In Section~\ref{section:problem},
we introduce basic notation and discuss the limitations of a compact \MINLP
formulation for \rcpp. Section~\ref{section:cg} presents a first Dantzig-Wolfe
decomposition of the \MINLP formulation. After introducing the concept of
circular and rectangular patterns, we extend the formulation from
Section~\ref{section:cg} in Section~\ref{section:pdw} to our final formulation
for \rcpp. Afterwards, we present a column generation algorithm to solve this
formulation, which uses an enumeration scheme introduced in
Section~\ref{section:enumeration}. Section~\ref{section:validbounds} shows how
to prove valid dual and primal bounds, even when difficult sub-problems cannot
be solved to optimality. Finally, in Section~\ref{section:experiments}, we
analyze the performance of our method on a large test set containing $800$
synthetic instances and nine real-world instances from the tube
industry. Section~\ref{section:conclusion} gives concluding remarks.

\section{Problem Statement}
\label{section:problem}

Consider a set $\typeset := \left\{1, \ldots, \ntypes\right\}$ for $\ntypes$
different types of rings and an infinite number of rectangles. In what follows,
we consider each rectangle to be of size $\width \in \R_+$ times
$\height \in \R_+$ and to be placed in the Euclidean plane such that the corners
are $(0,0)$, $(0,\width)$, $(\height,0)$, and $(\width, \height)$.

For each ring type $t \in \typeset$ we are given an internal radius
$\rint_t \in \R_+$ and an external radius $\rext_t \in \R_+$ such that
\begin{equation*}
  \rint_t \le \rext_t \le \min\{\width,\height\}.
\end{equation*}
Also, each ring type $t \in \typeset$ has a demand $\demand_t \in \Z_{+}$.
We assume without loss of generality that $\typeset$ is sorted such that
$\rext_1 \le \ldots \le \rext_\ntypes$. To simplify the notation, we denote by
$\nrings := \sum_{t \in \typeset} \demand_t$ the total number of individual
rings and denote by $\ringset := \{1, \ldots, \nrings\}$ the corresponding index
set. The function $\type : \ringset \rightarrow \typeset$ maps each individual
ring to its corresponding type.
By a slight abuse of notation, we identify with $\rint_i$ and $\rext_i$ the
internal and external radius of ring $i \in \ringset$, i.e.,
$\rext_i = \rext_{\type(i)}$ and $\rint_i = \rint_{\type(i)}$.

The task in \rcpp is to pack all rings in $\ringset$ into the smallest number of
rectangles. Rings must lie within the boundary of a rectangle and must not
intersect each other.
More precisely, a feasible solution to \rcpp can be encoded as a 3-tuple
$(c,x,y) \in \{1,\ldots,k\}^n \times \R^n \times \R^n$ where $(x_i,y_i)$ denotes
the center of ring $i \in \ringset$ inside rectangle $c_i \in \{1,\ldots,k\}$,
and an upper bound on the number of rectangles needed is given by $k \le n$. The
number of used rectangles is equal to the cardinality of
$\{c_1,\ldots,c_\nrings\}$, i.e., the number of distinct integer values. Rings
must not intersect the boundary of the rectangle, i.e.,
\begin{align}
  \rext_i \le x_i \le \width - \rext_i  & \fa i \in \ringset, \label{eq:minlp:boundsx}\\
  \rext_i \le y_i \le \height - \rext_i & \fa i \in \ringset. \label{eq:minlp:boundsy}
\end{align}
For a given 3-tuple $(c,x,y)$ we denote by
\begin{equation*}
  A(i) := \left\{(\tilde x, \tilde y) \in \R^2 \mid \rint_i < \norm{\mvec{\tilde x-x_i}{\tilde y-y_i}}_2 < \rext_i \right\}
\end{equation*}
the area occupied by ring~$i$ in rectangle~$c_i$. The non-overlapping condition
between different rings is equivalent to
\begin{equation} \label{eq:cip:nonoverlap}
  A(i) \cap A(j) = \emptyset
\end{equation}
for all $i \neq j$ with $c_i = c_j$.

\rcpp can be equivalently formulated as an \MINLP, see Pedroso et
al.~\cite{Pedroso2016}. The formulation consists of four different types of
variables:
\begin{itemize}
\item $(x_i,y_i) \in \R^2$, the center of ring $i \in \ringset$,
\item $z_c \in \{0,1\}$, a decision variable whether rectangle $c \in \{1,\ldots,k\}$ is used,
\item $w_{i,c} \in \{0,1\}$, a decision variable whether ring $i$ is directly placed in rectangle $c$, and
\item $u_{i,j} \in \{0,1\}$, a decision variable whether ring $i$ is directly placed in ring $j$.
\end{itemize}
We say that a ring $i \in \ringset$ is \emph{directly placed} in another ring
$j$ or rectangle $c$ if it is not contained in another, larger ring inside $j$
or $c$, respectively. Condition~\eqref{eq:cip:nonoverlap} can be modeled by the
constraints
\begin{align}
  && \norm{\mvec{x_i}{y_i} - \mvec{x_j}{y_j}}_2^2 &\ge (\rext_i + \rext_j)^2(w_{i,c} + w_{j,c} -1) && \fa i, j \in \ringset : i \neq j, \; c \in \{1,\ldots,k\}, \label{eq:minlp:nonoverlap1}\\
  && \norm{\mvec{x_i}{y_i} - \mvec{x_j}{y_j}}_2^2 &\ge (\rext_i + \rext_j)^2(u_{i,h} + u_{j,h} -1) && \fa i,j,h \in \ringset: i \neq j \wedge i \neq h \wedge j \neq h, \label{eq:minlp:nonoverlap2} \\
  && \norm{\mvec{x_i}{y_i} - \mvec{x_j}{y_j}}_2 &\le \rint_j - \rext_i + M_{i,j} (1-u_{i,j}) && \fa i, j \in \ringset : \rext_i \le \rint_j, \label{eq:minlp:nonoverlap3}
\end{align}
which are nonconvex nonlinear inequality
constraints. Inequality~\eqref{eq:minlp:nonoverlap1} guarantees that no two
rings~$i,j \in \ringset$ overlap when they are directly placed inside the same
rectangle~$c$, i.e., if~$w_{i,c} = w_{j,c} = 1$. Similarly,
inequality~\eqref{eq:minlp:nonoverlap2} ensures that no two rings intersect
inside another ring.
Constraint~\eqref{eq:minlp:nonoverlap3} ensures that if $u_{i,j} = 1$, then ring
$i$ is directly placed inside $j$ and does not intersect its
boundary. Otherwise, the constraint is disabled. An appropriate value for
$M_{i,j}$ to guarantee that the conditions of~\eqref{eq:minlp:nonoverlap3} are
satisfied is
\begin{equation*}
  M_{i,j} := \sqrt{(\width - \rext_i - \rext_j)^2 + (\height - \rext_i - \rext_j)^2},
\end{equation*}
which is the maximum distance between two rings $i,j \in \ringset$ in a $\width$
times $\height$ rectangle that do not overlap.

The full \MINLP model for \rcpp reads as follows:
\begin{subequations} \label{eq:minlp}
\begin{align}
  && \min \; & \sum_{c = 1}^k z_c \label{eq:minlp:obj}\\
  && \text{s.t.} \quad & w_{i,c} \le z_c && \fa i \in \ringset, \; c \in \{1,\ldots,k\} \label{eq:minlp:cons} \\
  &&                   & \sum_{c =1}^k w_{i,c} + \sum_{j \in \ringset} u_{i,j} = 1 && \fa i \in \ringset \label{eq:minlp:demand} \\
  &&                   & \eqref{eq:minlp:boundsx}, \eqref{eq:minlp:boundsy}, \eqref{eq:minlp:nonoverlap1}, \eqref{eq:minlp:nonoverlap2}, \eqref{eq:minlp:nonoverlap3} \nonumber \\
  &&                   & (x_i,y_i) \in \R^2 && \fa i \in \ringset \nonumber \\
  &&                   & z_c \in \{0,1\} && \fa c \in \{1,\ldots,k\} \nonumber \\
  &&                   & w_{i,c} \in \{0,1\} && \fa i \in \ringset, \; c \in \{1,\ldots,k\} \nonumber \\
  &&                   & u_{i,j} \in \{0,1\} && \fa i,j \in \ringset: i \neq j \nonumber
\end{align}
\end{subequations}
Finally, the objective function~\eqref{eq:minlp:obj} minimizes the total number
of rectangles used. Constraint~\eqref{eq:minlp:cons} guarantees that we can pack
a ring inside a rectangle if and only if the rectangle is used. Each ring needs
to be packed into another ring or a rectangle, which is ensured
by~\eqref{eq:minlp:demand}.

\begin{remark}\label{remark:rcppmax}\normalfont
  Formulation~\eqref{eq:minlp} can be adapted for maximizing the load of a
  single rectangle. Roughly speaking, we need to fix $z_c = 0$ for each $c > 1$
  and replace each variable $w_{i,c}$ by $w_{i}$, which then turns into an
  indicator whether ring $i \in \ringset$ has been packed or not. The objective
  function~\eqref{eq:minlp:obj} changes to
\begin{align*}
  \max \; \sum_{i \in \ringset} \alpha_i w_i,
\end{align*}
where $\alpha_i \in \R_+$ is a non-negative weight for each ring $i\in\ringset$.
\end{remark}

Unfortunately, general \MINLP solvers require a considerable amount of time to
solve Formulation~\eqref{eq:minlp} because it contains a large number of
variables, many nonconvex constraints, and much symmetry.
Even the smallest instance of our test set, containing $54$ individual rings,
could not be solved within several hours by \scip~3.2.1~\cite{SCIP} nor by
\baron~16.8.24~\cite{Kilinc2014,Sahinidis1996}. A common method used to improve
the solvability of geometric packing problems is to add symmetry breaking
constraints. However, strengthening~\eqref{eq:minlp} by adding the constraints
\begin{align}
  && z_c    &\ge z_{c+1} && \fa c \in \{1,\ldots,k-1\}, \label{eq:symmetry:cons1} \\
  && u_{i,j} &\ge u_{h,j} && \fa j \in \ringset, \; i < h: \type(i) = \type(h), \text{ and} \label{eq:symmetry:cons2} \\
  && w_{i,c} &\ge w_{j,c} && \fa c \in \{1,\ldots,k\}, \; i < j: \type(i) = \type(j), \label{eq:symmetry:cons3}
\end{align}
which break symmetry by ordering rectangles and rings of the same type, had no
impact on the solvability of \rcpp.

The reason for the poor performance is that~\eqref{eq:minlp} contains
$O(kn^2 + n^3)$ many difficult nonconvex constraints of the
form~\eqref{eq:minlp:nonoverlap1} and~\eqref{eq:minlp:nonoverlap2}. Even worse,
the binary variables in those constraints appear with a big-M, which is known to
result in weak linear programming relaxations.
Another difficulty is the symmetry in the model that is not eliminated
by~\eqref{eq:symmetry:cons1},~\eqref{eq:symmetry:cons2},
or~\eqref{eq:symmetry:cons3}. Each reflection or rotation of a feasible packing
of a rectangle and rings inside a ring yields another, equivalent, solution.

In the following, we present two formulations based on a Dantzig-Wolfe
decomposition to break the remaining symmetry in~\eqref{eq:minlp}. The first
formulation breaks the symmetry between rectangles. The second is an extension
of the first one and additionally breaks symmetry between rings of the same type
inside the rectangles and other rings.

To overcome the difficulty of the exponential number of variables in the
Dantzig-Wolfe decompositions, we use column generation to solve the continuous
relaxations of both formulations. The pricing problem of the first formulation
is the maximization version of~\eqref{eq:minlp}. The pricing problem of the
second formulation is a maximization version of the \cpp.
 
\section{Cutting Stock Reformulation}\label{section:cg}

Dantzig-Wolfe decomposition~\cite{Dantzig1960} is a classic solution approach
for structured \LPs. It can be used to decompose an \MILP into a \emph{master
  problem} and one or several \emph{sub-problems}. Advantages of the
decomposition are that it yields stronger relaxations if the sub-problems are
nonconvex and can aggregate equivalent sub-problems to reduce
symmetries~\cite{Lubbecke2005,Desaulniers2006}.

In this section, we apply a Dantzig-Wolfe decomposition to~\eqref{eq:minlp} and
obtain a reformulation that is similar to the one of the one-dimensional Cutting
Stock Problem~\cite{Vance1998}. The key idea is to reformulate \rcpp in order to
not assign rings explicitly to rectangles, but rather choose a set of feasible
rectangle packings to cover the demand for each ring type. The resulting
reformulation is a pure integer program (\IP) containing exponentially many
variables. We solve this reformulation via column generation.

We call a vector $F\in\Z_{+}^\ntypes$ \emph{packable}
if it is possible to pack $F_t$ many rings for all types $t \in \typeset$ together (and thus a
total number of $\sum_{t \in \typeset}F_t$ many rings) into a $\width\times\height$ rectangle.
A packable $F\in\Z_{+}^\ntypes$ corresponds to a feasible solution
of~\eqref{eq:minlp} when considering only a single rectangle.
Denote by
\begin{align*}
  \F := \{ F \in \Z_{+}^\ntypes \, : \, F \text{ is packable}\}
\end{align*}
the set of all feasible packings of rings into a rectangle. Note that
without loss of generality, we can bound $F_t$ by the demand $\demand_t$ for
each $t \in \typeset$. After applying Dantzig-Wolfe decomposition
to~\eqref{eq:minlp}, we obtain the following formulation, $\DW(\F)$:
\begin{subequations} \label{eq:dwmodel}
\begin{align}
  &&\min        \;    & \sum_{F \in \F} z_F \label{cg:reform1:obj} \\
  &&\text{s.t.} \quad & \sum_{F \in \F} F_t \cdot z_F \ge \demand_t && \fa t \in \typeset \label{cg:reform1:demand} \\
  &&                  & z_F \in \Z_{+} && \fa F \in \F
\end{align}
\end{subequations}
This formulation contains an exponential number of integer variables with
respect to $\ntypes$. Each variable $z_F$ counts how often $F$ is
chosen. Minimizing the total number of rectangles is equivalent
to~\eqref{cg:reform1:obj}. Inequality~\eqref{cg:reform1:demand} ensures that
each ring of type $t$ is packed at least $\demand_t$ many times.
In the following, we call the \LP relaxation of $\DW(\Fprime)$ for
$\Fprime \subseteq \F$ the \emph{restricted master problem} of $\DW(\F)$.

An advantage of~\eqref{eq:dwmodel} is that we do not explicitly assign rings to
positions inside some rectangles, like in~\eqref{eq:minlp}, nor distinguish
between rings of the same type. This breaks symmetry that arises from simply
permuting rectangles.  However, $\F$ is a priori unknown and its size is
exponential in the input size of an \rcpp instance.

A column generation algorithm is used to solve the LP relaxation of $\DW(\F)$ by
iteratively updating $\DW(\F)$ with improving columns corresponding to feasible
rectangular packings.  A feasible rectangular packing $F \in \F$ is an improving
column if the corresponding constraint
\begin{align*}
  \sum_{t \in \typeset} F_t \pi^*_t \le 1
\end{align*}
is violated by the dual solution $\pi^*$ of the restricted master problem
$\DW(\Fprime)$. The \LP relaxation of $\DW(\F)$ is solved to optimality when
$\sum_{t \in \typeset} F_t \pi^*_t \le 1$ holds for all $F \in \F \, \backslash \, \Fprime$. Otherwise, an
$F \in \F \, \backslash \, \Fprime$ with $\sum_{t \in \typeset} F_t \pi^*_t > 1$ is
added to $\Fprime$ and the process iterates until $\DW(\F)$ is solved to
optimality.

The most violated dual constraint is found by solving
\begin{equation}\label{prob:pricing1}
  \min \left\{ 1 - \sum_{t \in \typeset} \pi^*_t F_t : F \in \Z_{+}^\ntypes, \, F \text{ is packable}\right\}.
\end{equation}
This problem is a maximization variant of \rcpp for which we need to find a
subset of rings such that they can be packed into a single rectangle, see
Remark~\ref{remark:rcppmax}. The objective gain of each ring of type $t$ is
exactly $\pi^*_t$. The \LP relaxation of~\eqref{eq:dwmodel} is solved to
optimality if the solution value of~\eqref{prob:pricing1} is
non-negative. Otherwise, we find an improving column that, after it has been
added, might decrease the objective value of the restricted master problem
$\DW(\Fprime)$.

As explained in Remark~\ref{remark:rcppmax},~\eqref{prob:pricing1} can be
modeled as a nonconvex \MINLP that contains the selection and positioning of
rings in a rectangle. Unfortunately, solving this problem is very difficult for
a general \MINLP solver. In our experiments none of the resulting pricing
problems~\eqref{prob:pricing1} could be solved to optimality
or even generate an improving column for $\DW(\Fprime)$ in two hours. For this
reason, Formulation~\eqref{eq:dwmodel} is not suitable for solving \rcpp to
global optimality.
 
\section{Pattern-based Dantzig-Wolfe Decomposition using One-level Packings}
\label{section:pdw}

The main drawback of~\eqref{eq:dwmodel} is that the resulting pricing
problems~\eqref{prob:pricing1} are intractable. This is due to two major
difficulties, i) the positioning of rings inside a rectangle and ii) the
combinatorial decisions of how to put rings into other rings. Together, they
make the sub-problems much more difficult to solve than the \IP master
problem. The recursive decisions of packing rings into each other introduces
much symmetry to~\eqref{prob:pricing1}. Rings of the same type, and all rings
packed inside those, can be swapped inside a rectangle and yield an equivalent
solution. This symmetry appears in each recursion level of the sub-problems. See
Figure~\ref{fig:symmetricsol} for an example of this kind of symmetry.

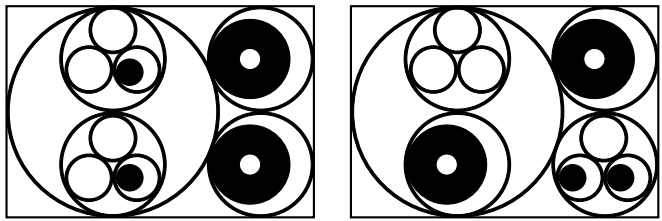
\begin{figure}[ht]
  \centering
  \centering
\begin{minipage}{0.30\textwidth}
  \centering
  \begin{tikzpicture}[scale=0.35]
    \draw[thick] (0,0) -- (11.55,0) -- (11.55,8.0) -- (0, 8.0) -- cycle;
    \foreach \id/\x/\y/\R/\r in
             {
               1/4.0/4.0/4.0/3.9,
               2/4.0/6.0/2.0/1.9,
               2/4.0/2.0/2.0/1.9,
               3/4.0/7.1/0.9/0.8,
               3/3.1/5.6/0.9/0.8,
               3/4.9/5.6/0.9/0.8,
               3/4.0/3.0/0.9/0.8,
               3/3.1/1.5/0.9/0.8,
               3/4.9/1.5/0.9/0.8,
               2/9.55/6.0/2.0/1.9,
               2/9.55/2.0/2.0/1.9,
               4/9.15/2.0/1.5/0.4,
               4/9.15/6.0/1.5/0.4,
               5/4.63/1.5/0.5/0.0,
               5/4.63/5.5/0.5/0.0
             }
             \draw[fill=black,even odd rule] (\x,\y) circle (\R) circle (\r);
  \end{tikzpicture}
\end{minipage}
\begin{minipage}{0.30\textwidth}
  \centering
  \begin{tikzpicture}[scale=0.35]
    \draw[thick] (0,0) -- (11.55,0) -- (11.55,8.0) -- (0, 8.0) -- cycle;
    \foreach \id/\x/\y/\R/\r in
             {
               1/4.0/4.0/4.0/3.9,
               2/4.0/6.0/2.0/1.9,
               2/4.0/2.0/2.0/1.9,
               3/4.0/7.1/0.9/0.8,
               3/3.1/5.6/0.9/0.8,
               3/4.9/5.6/0.9/0.8,
               3/9.5/3.0/0.9/0.8,
               3/8.6/1.5/0.9/0.8,
               3/10.4/1.5/0.9/0.8,
               2/9.55/6.0/2.0/1.9,
               2/9.55/2.0/2.0/1.9,
               4/3.6/2.0/1.5/0.4,
               4/9.15/6.0/1.5/0.4,
               5/8.33/1.5/0.5/0.0,
               5/10.13/1.5/0.5/0.0
             }
             \draw[fill=black,even odd rule] (\x,\y) circle (\R) circle (\r);
  \end{tikzpicture}
\end{minipage}
   \captionsetup{justification=centering}
  \caption{Two equivalent feasible packings obtained by swapping rings packed
    inside other rings.}
  \label{fig:symmetricsol}
\end{figure}

The main idea of the following reformulation is to break this type of symmetry
and shift the recursive decisions from the sub-problem to the master
problem. This helps balance the complexity of both problems, which is crucial when
using a column generation algorithm.

In the following, we introduce the concept of \emph{circular} and
\emph{rectangular patterns}. These patterns describe possible packings of
circles into rings or rectangles. The circles act like
placeholders. Specifically, the circles just describe what type of rings might
be placed in the circles, but not how these rings are filled with other
rings. After choosing one pattern we are able to choose other patterns that fit
into the circles of the selected one. The recursive part of \rcpp boils down to
a counting problem of patterns and minimizing the total number of rectangles.

\subsection{Circular Patterns}
\label{section:circularpatterns}

A \emph{circle} of type $t \in \typeset$ is a ring with external radius
$\rext_t$ and inner radius $\rint_t = 0$. This means that neither circles nor
rings can be put into a circle.
Similarly to the definition of elements in $\F$, we call a tuple $(t,P)\in
\typeset \times \Z_{+}^\ntypes$ a \emph{circular pattern} if it is possible to
pack~$P_1$ many circles of type~$1$, $P_2$ many circles of type~$2$, \ldots,
$P_\ntypes$ many circles of type~$\ntypes$ into a larger ring of type~$t$.
As an example, $(3,(2,1,0))$ is a circular pattern with an outer ring
of type three which contains two circles with external radius $\rext_1$ and one
circle with radius $\rext_2$. Since $\rext_3$ can not be pack within any ring, the final index will
always be 0.  However, it is included in the definition of circular patterns for
completeness.
Figure~\ref{fig:patterns} shows all possible packings for three different ring
types.

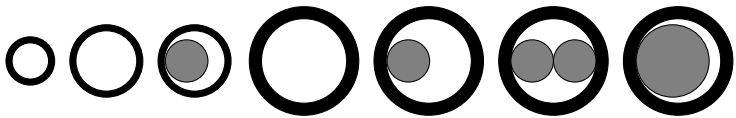
\begin{figure}[ht]
  \centering
  \begin{tikzpicture}[scale = 0.4]

   \foreach \id/\x/\y/\R/\r in
           {
             1/0.0/-2.0/0.8/0.6,
             2/2.5/-2.0/1.2/1.0,
             2/5.4/-2.0/1.2/1.0,
             3/9.0/-2.0/1.8/1.4,
             3/13.1/-2.0/1.8/1.4,
             3/17.2/-2.0/1.8/1.4,
             3/21.3/-2.0/1.8/1.4
           }
           \draw[fill=black,even odd rule] (\x,\y) circle (\R) circle (\r);

   \foreach \id/\x/\y/\R in
           {
             1/5.14/-2.0/0.7,
             1/12.43/-2.0/0.7,
             1/16.5/-2.0/0.7,
             1/17.9/-2.0/0.7,
             2/21.12/-2.0/1.2
           }
           \draw[fill=gray,even odd rule] (\x,\y) circle (\R);
\end{tikzpicture}
   \captionsetup{justification=centering}
  \caption{All possible circular patterns for three different ring types: \\
    $(1,(0,0,0)),(2,(0,0,0)),(2,(1,0,0)),(3,(0,0,0)),(3,(1,0,0)),(3,(2,0,0)),(3,(0,1,0))$}
  \label{fig:patterns}
\end{figure}

Using the definition of circular patterns we decompose a packing of rings into a
ring $R$. Each ring that is directly placed into $R$, i.e., is not contained in
another larger ring, is replaced by a circle with the same external radius. We
apply the decomposition to each replaced ring recursively. As a result, ring $R$
decomposes into a set of circular patterns. Starting from these circular
patterns, we can reconstruct $R$ by recursively replacing circles in a pattern
with other circular patterns. We denote by
\begin{align*}
  \cp := \{ (t,P) \in \typeset \times \Z_{+}^\ntypes \mid (t,P) \text{ is a circular pattern} \}
\end{align*}
the set of all possible circular patterns. The previously described
decomposition shows that any recursive packing of rings can be constructed by
using circular patterns of~$\cp$.

In general, the cardinality of $\cp$ is exponential in $\ntypes$ and depends on
the internal and external radii of the rings. For example, increasing the inner
radius of the largest ring will lead to many more possibilities to pack circles
into this ring. In contrast, decreasing external radii results in fewer circular
patterns.

Figure~\ref{fig:patterns} shows that there are circular patterns that are
dominated by others, e.g., the third pattern dominates second one. Using
dominated circular patterns would leave some unnecessary free space in some
rectangles, which can be avoided in an optimal solution. In
Section~\ref{section:enumeration} we discuss this domination relation and
present an algorithm to compute all non-dominated circular patterns.

\subsection{Rectangular Patterns}
\label{section:rectpatterns}

In the following reformulation of \rcpp, we use circular patterns to shift the
decisions of how to put rings into other rings to the master problem. Similar to
a circular pattern, we call $P \in \Z_{+}^\typeset$ a \emph{rectangular pattern}
if and only if $P_t$ many circles with radius $\rext_t$ for each
$t \in \typeset$ can be packed together into a rectangle of size $\width$ times
$\height$. Let
\begin{align*}
  \rp := \{ P \in \Z_{+}^\ntypes \mid P \text{ is a rectangular pattern} \}
\end{align*}
be the set of all rectangular patterns. As in Formulation~\eqref{eq:dwmodel},
only the number of packed circles matters, not their position in the rectangular
pattern.

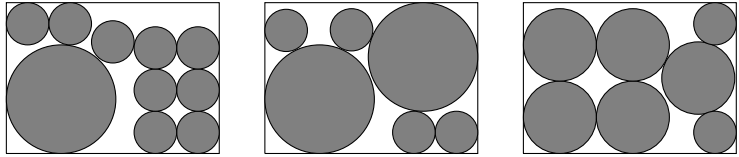
\begin{figure}[ht]
  \centering
  \begin{tikzpicture}[scale=0.4]
  \draw (-5.5,-2.5) --++ (7,0) --++ (0,5) --++ (-7,0) --++ (0,-5);
  \draw[fill=gray] (-3.7,-0.7) circle (1.8);
  \draw[fill=gray] (0.8,-1.8) circle (0.7);
  \draw[fill=gray] (0.8,-0.4) circle (0.7);
  \draw[fill=gray] (0.8, 1.0) circle (0.7);
  \draw[fill=gray] (-0.6,-1.8) circle (0.7);
  \draw[fill=gray] (-0.6,-0.4) circle (0.7);
  \draw[fill=gray] (-0.6, 1.0) circle (0.7);
  \draw[fill=gray] (-4.8, 1.8) circle (0.7);
  \draw[fill=gray] (-3.4, 1.8) circle (0.7);
  \draw[fill=gray] (-2.0, 1.2) circle (0.7);

  \draw (3,-2.5) --++ (7,0) --++ (0,5) --++ (-7,0) --++ (0,-5);
  \draw[fill=gray] (4.8,-0.7) circle (1.8);
  \draw[fill=gray] (8.2,0.7) circle (1.8);
  \draw[fill=gray] (3.7,1.58) circle (0.7);
  \draw[fill=gray] (5.85,1.6) circle (0.7);
  \draw[fill=gray] (9.3,-1.8) circle (0.7);
  \draw[fill=gray] (7.9,-1.8) circle (0.7);

  \draw (11.5,-2.5) --++ (7,0) --++ (0,5) --++ (-7,0) --++ (0,-5);
  \draw[fill=gray] (12.7,-1.3) circle (1.2);
  \draw[fill=gray] (15.1,-1.3) circle (1.2);
  \draw[fill=gray] (17.25, -0.0) circle (1.2);
  \draw[fill=gray] (12.7, 1.1) circle (1.2);
  \draw[fill=gray] (15.1, 1.1) circle (1.2);
  \draw[fill=gray] (17.8, 1.8) circle (0.7);
  \draw[fill=gray] (17.8, -1.8) circle (0.7);
\end{tikzpicture}
   \caption{Three different rectangular patterns:\\ $(9,0,1),(4,0,2),(2,5,0)$}
  \label{fig:rectpatterns}
\end{figure}

In contrast to verifying if $(t,P) \in \typeset \times \Z_{+}^\ntypes$ is a
circular pattern, checking whether $P \in \Z_{+}^\ntypes$ is a rectangular
pattern---a classical \cpp---can be much more difficult. Typically, many more
circles fit into a rectangle than into a ring.  This results in a large number
of circles that need to be considered in a verification problem, which is in
practice difficult to solve.

\subsection{Exploiting Recursion}
\label{subsection:improvedreformulation}

Using the circle and rectangle patterns described above, we develop a
pattern-based Dantzig-Wolfe decomposition for \rcpp with nonlinear sub-problems.
Instead of placing rings explicitly into each other, we use patterns to remodel
the recursive part. The key idea is that circles, inside rectangular or circular
patterns, are replaced by circular patterns with the same external radius. These
circular patterns contain circles that can be replaced by other circular
patterns.
More precisely, after choosing a rectangular pattern $P \in \rp$, it is possible
to choose $P_t$ circular patterns of the form $(t,P') \in \cp$, which can be
placed into $P$. Again, for each~$P'$ we can choose $P'_{t'}$ many circular
patterns of the form $(t', P'')$, which can be placed in $P'$. This process can
continue until the smallest packable circle is considered. The recursive
structure of the \rcpp---the placement of rings into other rings---is modeled by
counting the number of used rectangular and circular patterns.

Figure~\ref{fig:cpdecomposition} illustrates this idea. A circular pattern
replaces a circle if there is an arrow from the pattern to the circle. Each
circle of a pattern can be used as often as the pattern is used. It follows that
the number of outgoing edges of a circular pattern is equal to the number of
uses of the pattern. The combinatorial part of \rcpp reduces to adding edges
from circular patterns to circles.

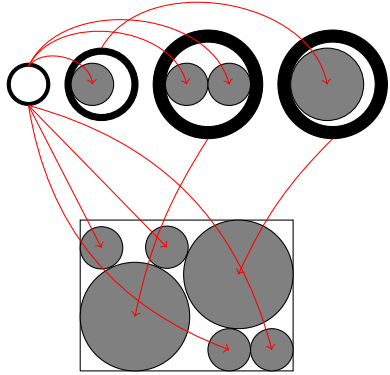
\begin{figure}[ht]
  \centering
  \begin{tikzpicture}[scale = 0.4]

  \foreach \id/\x/\y/\R/\r in
  {
    1/11.3/-2.0/0.7/0.6,
    2/13.7/-2.0/1.2/1.0,
    3/17.2/-2.0/1.8/1.4,
    3/21.3/-2.0/1.8/1.4
  }
  \draw[fill=black,even odd rule] (\x,\y) circle (\R) circle (\r);

  \foreach \id/\x/\y/\R in
  {
    1/13.4/-2.0/0.7,
    1/16.5/-2.0/0.7,
    1/17.9/-2.0/0.7,
    2/21.12/-2.0/1.2
  }
  \draw[fill=gray,even odd rule] (\x,\y) circle (\R);

  \newcommand*{\xshift}{10}
  \newcommand*{\yshift}{-9}
  \draw (3 + \xshift,-2.5 + \yshift) --++ (7,0) --++ (0,5) --++ (-7,0) --++ (0,-5);
  \foreach \x/\y/\R in
  {
    4.8 / -0.7/ 1.8,
    8.2 /  0.7/ 1.8,
    3.7 / 1.58/ 0.7,
    5.85/  1.6/ 0.7,
    9.3/  -1.8/ 0.7,
    7.9/  -1.8/ 0.7
  }
  \draw[fill=gray] (\x + \xshift,\y + \yshift) circle (\R);

  \draw[red,->] (11.3,-2.7) -- (3.7 + \xshift, 1.58 + \yshift);
  \draw[red,->] (11.3,-2.7) -- (5.85 + \xshift, 1.6 + \yshift);
  \draw[red,->] (11.3,-2.7) to [bend left=30] (9.3  + \xshift, -1.8 + \yshift);
  \draw[red,->] (11.3,-2.7) to [bend right=30](7.9  + \xshift, -1.8 + \yshift);

  \draw[red,->] (17.2,-3.8) to [bend right=10] (4.8  + \xshift, -0.7 + \yshift);
  \draw[red,->] (21.3,-3.8) to [bend right=10] (8.2  + \xshift, 0.7 + \yshift);

  \draw[red,->] (13.7,-0.8) to [bend left=70] (21.12, -2.0);

  \draw[red,->] (11.3,-1.4) to [bend left=70] (13.4,-2.0);

  \draw[red,->] (11.3,-1.4) to [bend left=70] (16.5,-2.0);
  \draw[red,->] (11.3,-1.4) to [bend left=70] (17.9,-2.0);
\end{tikzpicture}
   \caption{An example with four circular patterns and a rectangular pattern
    showing how patterns are used to model the combinatorial part of \rcpp. Each
    line connects a circular pattern to a circle. The number of outgoing edges
    is equal to the number of rings that are used.}
  \label{fig:cpdecomposition}
\end{figure}

Following this idea, we introduce integer variables
\begin{itemize}
\item $z_{C} \in \Z_{+}$ for each circular pattern $C \in \cp$ and
\item $z_{P} \in \Z_{+}$ for each rectangular $P \in \rp$
\end{itemize}
in order to count the number of used circular and rectangular patterns (in the
pattern-based formulation). We reformulate \rcpp by the following \IP
formulation $\PDW(\rp)$, which is similar to the multi-stage cutting stock
formulation that has been presented by Muter et al.~\cite{Muter2013}.
\begin{subequations} \label{eq:pdwmodel}
\begin{align}
  && \min \sum_{P \in \rp} z_P \label{eq:pdwmodel:obj} \quad\,\\
  && \text{s.t.} \sum_{C = (t,P) \in \cp} z_C &\ge \demand_t && \fa t \in \typeset \label{eq:pdwmodel:demand} \\
  && \sum_{P \in \rp} P_t \cdot z_P + \sum_{C = (t',P) \in \cp} P_t \cdot z_C &\ge \sum_{C = (t,P) \in \cp} z_C && \fa t \in \typeset \label{eq:pdwmodel:recursive} \\
  && z_C & \in \Z_{+} && \fa C \in \cp \\
  && z_P & \in \Z_{+} && \fa P \in \rp
\end{align}
\end{subequations}
Objective~\eqref{eq:pdwmodel:obj} minimizes the total number of used
rectangles. Constraint \eqref{eq:pdwmodel:demand} ensures that the demand for
each ring type is satisfied. The recursive decisions how to place rings into
each other are implicitly modeled by~\eqref{eq:pdwmodel:recursive}.
Each selection of a pattern $P \in \rp$ or $(t',P) \in \cp$ allows us to choose
$P_t$ circular patterns of the type $t$.
Note that at least one rectangular pattern needs to be selected before circular
patterns can be packed. This is true because the largest ring only fits into a
rectangular pattern.
To aid in understanding the formulation of $\PDW(\rp)$, a small
example based on the circular and rectangular patterns in Figures
\ref{fig:patterns} and \ref{fig:rectpatterns} respectively is presented in
Example~\ref{example:pdwmodel}.

\begin{example}\label{example:pdwmodel}\normalfont
  Let $\{C_1,\ldots,C_7\}$ be the set of circular patterns of
  Figure~\ref{fig:patterns} and $\{P_1,P_2,P_3\}$ be the subset of rectangular
  patterns of Figure~\ref{fig:rectpatterns}. The patterns are labeled from left
  to right. Problem~\eqref{eq:pdwmodel} then reads as
  \begin{subequations}
  \begin{align}
     &&    \min z_{P_1} + z_{P_2} + z_{P_3} \\
   \text{s.t.} && z_{C_1} & \ge D_1 \label{eq:example:demand1} \\
    &&              z_{C_2} + z_{C_3} & \ge D_2 \label{eq:example:demand2} \\
    &&              z_{C_4} + z_{C_5} + z_{C_6} + z_{C_7} & \ge D_3 \label{eq:example:demand3} \\
    &&              z_{C_3} + z_{C_5} + 2 z_{C_6} + 9 z_{P_1} + 4 z_{P_2} + 2 z_{P_3} & \ge z_{C_1} \label{eq:example:recursive1} \\
    &&              z_{C_7} + 5 z_{P_3} & \ge z_{C_2} + z_{C_3} \label{eq:example:recursive2} \\
    &&              z_{P_1} + 2 z_{P_2} & \ge z_{C_4} + z_{C_5} + z_{C_6} + z_{C_7} \label{eq:example:recursive3} \\
    &&              z_{C_i} &\in \Z_{+} \quad \fa i \in \{1,\ldots,7\} \\
    &&              z_{P_i} &\in \Z_{+} \quad \fa i \in \{1,2,3\}
  \end{align}
  \end{subequations}
  Constraints \eqref{eq:example:demand1}--\eqref{eq:example:demand3} ensure that
  the demand for each ring type is satisfied. The left-hand side of constraints
  \eqref{eq:example:demand1}--\eqref{eq:example:demand3} only contain columns
  corresponding to the circular patterns of ring type 1 to 3 respectively. The
  constraints \eqref{eq:example:recursive1}--\eqref{eq:example:recursive3}
  model the recursive structure of the problem. Columns corresponding to
  circular patterns for ring types 1 to 3 are observed on the right-hand side of
  constraints \eqref{eq:example:recursive1}--\eqref{eq:example:recursive3}
  respectively. The left-hand side of constraints
  \eqref{eq:example:recursive1}--\eqref{eq:example:recursive3} contain columns
  corresponding to rectangular and circular patterns that pack the ring type
  represented on the right-hand side of the respective constraints.
\end{example}

A drawback of~\eqref{eq:pdwmodel} is the exponential number of rectangular and
circular pattern variables.
To address this difficulty, we develop a column enumeration algorithm to compute
all (relevant) circular patterns used in~\eqref{eq:pdwmodel}, which is presented
in Section~\ref{section:enumeration}.
We observed in our experiments that for many instances this algorithm
successfully enumerates all circular patterns in a reasonable amount of time.

Since the size of the rectangles is much larger than the external radii $\rext$,
it is intractable to enumerate all rectangular patterns. To overcome this
difficulty, we use a column generation approach to solve the \LP relaxation
of~\eqref{eq:pdwmodel} that dynamically generates rectangular patterns
variables. We call the \LP relaxation of $\PDW(\rp')$ the restricted master
problem of~\eqref{eq:pdwmodel} for a subset of rectangular patterns $\rp'
\subseteq \rp$. In order to find an improving column for $\PDW(\rp')$, we solve
a weighted \cpp for a single rectangle.

More precisely, let $\lambda \in \R_+^\ntypes$ be the non-negative vector of
dual multipliers for Constraints~\eqref{eq:pdwmodel:recursive} after solving the
\LP relaxation of $\PDW(\rp')$ for the current set of rectangular patterns
$\rp'\subset\rp$. To compute a rectangular pattern with negative reduced cost
we solve
\begin{equation} \label{eq:pdwmodel:pricing}
  \min_{P \in \rp \, \backslash \, \rp'} \left\{1 - \sum_{t \in \typeset} \lambda_t P_t\right\},
\end{equation}
which can be modeled as a weighted \cpp for a single rectangle. Let $P^*$ be an
optimal solution to~\eqref{eq:pdwmodel:pricing}.  If $1 - \sum_{t \in \typeset}
\lambda_t P^*_t$ is negative, then $P^*$ is an improving rectangular pattern,
whose corresponding variable needs to be added to the restricted master problem
of $\PDW(\rp')$. Otherwise, the \LP relaxation of~\eqref{eq:pdwmodel} is solved
to optimality.

The pricing problem is $\mathcal{NP}$-hard~\cite{Lenstra1979} and difficult to
solve in practice. The number of variables in this problem depends on the number
of different ring types $\ntypes$ and on the demand vector $\demand$. When
solving~\eqref{eq:pdwmodel:pricing} we need to consider the index set of
individual circles
\begin{align*}
  \{ i^1_1, \ldots, i^1_{\demand_1}, i^2_{1}, \ldots, i^2_{\demand_2}, \ldots, i^\ntypes_1, \ldots, i^\ntypes_{\demand_\ntypes} \}
\end{align*}
containing $\demand_t$ indices that correspond to circles with radius $\rext_t$
for each $t \in \typeset$. The number of copies for type $t$ can be reduced to
$\min \{\demand_t, \floor{\frac{\pi (\rext_t)^2}{\width \height}} \}$, which is
an upper bound on the number of rings of type $t$ in a $\width$ times $\height$
rectangle.
For simplicity, denote with $\circleset$ the index set of all
individual circles.  Let $\rext_i$ be the external radius and $\type(i)$ the
type of circle $i \in \circleset$. The circle packing problem formulation
of~\eqref{eq:pdwmodel:pricing} reads then as
\begin{subequations} \label{eq:pdwmodel:pricingformulation}
\begin{align}
  && \min \; & 1 - \sum_{i \in \circleset} \lambda_{\type(i)} z_i \\
  \text{s.t.} && \norm{\mvec{x_i}{y_i} - \mvec{x_j}{y_j}}_2^2 &\ge (\rext_i + \rext_j)^2(z_{i} + z_{j} -1) && \fa i, j \in \circleset : i \neq j \\
  && \rext_i &\le x_i \le \width - \rext_i && \fa i \in \circleset \\
  && \rext_i &\le y_i \le \height - \rext_i && \fa i \in \circleset \\
  && z_i & \in \{0,1\} && \fa i \in \circleset
\end{align}
\end{subequations}
where $(x_i,y_i)$ is the center of circle $i \in \circleset$ and $z_i$ the
decision variable whether circle $i$ has been packed, i.e., $z_i = 1$.

After solving the continuous relaxation of~\eqref{eq:pdwmodel}, it might happen
that $z_P^*$ is fractional for a rectangular pattern $P \in \rp'$. In this case
branching is required to ensure global optimality.
A general branching strategy has been introduced by~\cite{Vanderbeck1996}, which
has been successfully used in a branch-and-price algorithm for the
one-dimensional cutting stock problem~\cite{Vance1998}. This branching rule can
be applied when solving~\eqref{eq:pdwmodel}, however, it increases
the complexity of the pricing problems~\eqref{eq:pdwmodel:pricing}.
Since~\eqref{eq:pdwmodel:pricing} is very difficult to solve---the vast majority of
pricing problems cannot be solved to optimality in the root node---employing
branch-and-price is deemed impractical. As such, Section~\ref{section:price-and-verify} presents a \emph{price-and-verify} algorithm,
which builds upon price-and-branch, as a practical method for
solving the \rcpp.  With techniques discussed in
  Section~\ref{section:validbounds} this results in an algorithm that is able to
  prove global optimality for many \rcpp instances.

\subsection{Strength of Dantzig-Wolfe reformulations}
\label{subsection:strengthFormulations}

In the following, we show that the two presented formulations~\eqref{eq:dwmodel}
and~\eqref{eq:pdwmodel} provide the same \LP bound.

\begin{theorem} \normalfont \label{theorem:lprelax} Let $\LP(DW)$ and $\LP(PDW)$
  be the value of the \LP relaxation of~\eqref{eq:dwmodel}
  and~\eqref{eq:pdwmodel}, respectively. Then $\LP(PDW) = \LP(DW)$.
\end{theorem}
\begin{proof}
  To show the equality of the \LP bounds, we consider an ``extended''
  Dantzig-Wolfe reformulation.  As a generalization of the variables $z_F$ from
  \eqref{eq:dwmodel}, where $F\in\F$ encodes a rectangle of recursively packable
  rings, and $z_P$ from \eqref{eq:pdwmodel}, where $P\in\rp$ encodes
  a rectangular pattern,
  we introduce variables for ``mixed'' rectangles that may contain both rings and circles. Let $z_{(F,P)}$ be the number of such
  ``mixed'' rectangles, and denote with
  \begin{equation*}
    \mrp = \{ (F,P) \in \Z_{+}^\ntypes \times \Z_{+}^\ntypes : (F,P) \text{ is packable } \}
  \end{equation*}
  the set of mixed rectangles.  Here, $(F,P)$ is packable if
  $F_t$ rings of type $t$ and $P_t$ circles of type $t$ can be packed into one
  rectangle without overlap, collectively over all $t$. A ring may contain smaller rings and circles, but
  circles cover their full interior, see Figure~\ref{fig:mixedrectangle} for an
  example.

  \begin{figure}[ht]
    \centering
    \centering
\begin{tikzpicture}[scale=0.4]
  \draw[thick,draw=black] (0,0) -- (8.0,0) -- (8.0,7.0) -- (0, 7.0) -- cycle;
  \foreach \id/\x/\y/\R/\r in
           {
             8/4.721279/6.098191/0.9/0.6,
             9/0.9/2.317016/0.9/0.6,
             12/5.551494/2.4/2.4/2.2,
             13/2.14/0.9/0.9/0.6,
             12/4.6/2.8/1.2/1.0,
             13/4.6/2.8/0.9/0.6,
             13/6.7/3.0/0.9/0.6
           }
           \draw[fill=black,draw=black,even odd rule] (\x,\y) circle (\R) circle (\r);
  \foreach \id/\x/\y/\R/\r in
           {
             2/6.8/5.8/1.2/1.0,
             3/2.00128/5.0/2.0/0.4,
             13/6.0/1.2/0.9/0.6
           }
           \draw[fill=gray,draw=black] (\x,\y) circle (\R);
\end{tikzpicture}
     \captionsetup{justification=centering}
    \caption{An example for a mixed rectangle packed with rings and circles.}
    \label{fig:mixedrectangle}
  \end{figure}
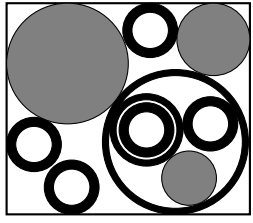

  Then define the \LP
  \begin{subequations} \label{eq:mdwmodel}
    \begin{align}
      && \min \sum_{(F,P) \in \mrp} z_{(F,P)} \label{eq:mdwmodel:obj} \quad\,\\
      && \text{s.t.} \sum_{(F,P) \in \mrp} F_t \cdot z_{(F,P)} + \sum_{C = (t,P) \in \cp} z_C &\ge \demand_t && \fa t \in \typeset \label{eq:mdwmodel:demand} \\
      && \sum_{(F,P) \in \mrp} P_t \cdot z_{(F,P)} + \sum_{C \in \cp} P_t \cdot z_C &\ge \sum_{C = (t,P) \in \cp} z_C && \fa t \in \typeset \label{eq:mdwmodel:recursive} \\
      && z_C & \geq 0 && \fa C \in \cp \\
      && z_{(F,P)} & \geq 0 && \fa (F,P) \in \mrp
    \end{align}
  \end{subequations}
  and let $Z^*$ be its optimal objective value. By construction,
  \eqref{eq:mdwmodel} contains the feasible solutions of \eqref{eq:dwmodel}
  (setting $z_C=0$ and $z_{(F,P)}=0$ for all $P\not=0$) and \eqref{eq:pdwmodel}
  (setting $z_{(F,P)}=0$ for all $F\not=0$). Hence, $Z^* \leq \LP(DW)$ and $Z^*
  \leq \LP(PDW)$. We prove the converse by optimality-preserving exchange
  arguments. First, given a solution with rectangles that contain rings, these
  rings can gradually be replaced by circles and circular patterns. Second,
  given a solution with rectangles that contain circles, we are guaranteed to
  find a circular pattern that can take its place. Formally, the proof reads as
  follows.

  1. Suppose $\LP(PDW) > Z^*$ and let $z^*$ be an optimal solution to
  \eqref{eq:mdwmodel} such that the total number of rings used in the rectangles
  of the support,
  \begin{equation*}
    \mu(z) = \sum_{ (F,P) \in\mrp } \sum_{t=1}^\ntypes F_t \cdot z_{(F,P)}
  \end{equation*}
  is minimal.  By assumption, $\mu(z^*) > 0$, and we can choose $(\tilde F, \tilde
  P)$ and $\tilde t$ such that $z^*_{(\tilde F,\tilde P)} >0$ and $\tilde
  F_{\tilde t} \geq 1$.  Choose ring type $\tilde t$ to have smallest external radius $\rext_{\tilde t}$. Then this ring can only contain circles, and with these circles it forms a circular pattern $\tilde C = (\tilde t, P') \in \cp$. Replacing the ring of type $\tilde t$ by a circle with same external radius
  gives the packable tuple $(\hat F, \hat P) := (\tilde F - e_{\tilde t}, \tilde
  P + e_{\tilde t})$.  Then
  define a new solution $\hat z$ via
  \begin{align*}
    \hat z_{(F,P)} &:= \begin{cases}
    0 & \text{ if } (F,P) = (\tilde F, \tilde P),\\
    z^*_{(\hat F,\hat P)} + z^*_{(\tilde F,\tilde P)} & \text{ if } (F,P) = (\hat F, \hat P),\\
    z^*_{(F,P)} & \text{ otherwise},
    \end{cases}
    \intertext{for all $(F,P) \in \mrp$, and}
    \hat z_{C} &:= \begin{cases}
    z^*_{\tilde C} + z^*_{(\tilde F,\tilde P)} & \text{ if } C = \tilde C,\\
    z^*_{C} & \text{ otherwise},
    \end{cases}
  \end{align*}
  for all $C \in \cp$.  By construction, $\hat z$ is feasible for
  \eqref{eq:mdwmodel} and has identical objective function value.  However,
  $\mu(\hat z) = \mu(z^*) - z^*_{(\tilde F,\tilde P)} < \mu(z^*)$, contradicting
  the minimality assumption.

  2. Suppose $\LP(DW) > Z^*$ and let $z^*$ be an optimal solution to
  \eqref{eq:dwmodel} such that $\nu(z) = \sum_{C \in\cp} z_C$ is minimal.  By
  assumption, $\nu(z^*) > 0$; otherwise, $z^*$ could be transformed into an
  optimal solution of $\LP(DW)$ by removing all circles from each
  rectangle. Hence, we can choose a circular pattern $\tilde C = (\tilde t, P')$
  with $z^*_{\tilde C} > 0$ and largest external radius $\rext_{\tilde t}$.

  Now consider Constraint \eqref{eq:mdwmodel:recursive} for $t=\tilde
  t$. Because $\sum_{C \in \cp} P_{\tilde t} \cdot z_C$ must be zero, there
  exists an $(\tilde F,\tilde P)$ such that $z^*_{(\tilde F,\tilde P)} > 0$ and
  $\tilde P_{\tilde t} \geq 1$.  Replacing one ring of type $\tilde t$ by the
  circular pattern $\tilde C = (\tilde t,P')$ gives the packable tuple $(\hat F,
  \hat P) := (\tilde F + e_{\tilde t}, \tilde P - e_{\tilde t} + P')$.  Finally,
  we define a new solution $\hat z$ via
  \begin{align*}
    \hat z_{(F,P)} &:= \begin{cases}
    0 & \text{ if } (F,P) = (\tilde F, \tilde P),\\
    z^*_{(\hat F,\hat P)} + z^*_{(\tilde F,\tilde P)} & \text{ if } (F,P) = (\hat F, \hat P),\\
    z^*_{(F,P)} & \text{ otherwise},
    \end{cases}
    \intertext{for all $(F,P) \in \mrp$, and}
    \hat z_{C} &:= \begin{cases}
    0 & \text{ if } C = \tilde C,\\
    z^*_{C} & \text{ otherwise},
    \end{cases}
  \end{align*}
  for all $C \in \cp$.  By construction, $\hat z$ is feasible for
  \eqref{eq:mdwmodel} and has identical objective function value.  However,
  $\nu(\hat z) = \nu(z^*) - z^*_{\tilde C} < \nu(z^*)$, contradicting the
  minimality assumption.
  \qed
\end{proof}

While \eqref{eq:pdwmodel} does not improve
the strength of the \LP relaxation, compared to \eqref{eq:dwmodel}, the
advantage of~\eqref{eq:pdwmodel} is that it breaks the symmetry of the
combinatorial part of the \rcpp, i.e., packing rings into rings, on each
recursion level. Applying the pattern-based Dantzig-Wolfe reformulation makes
deciding how to pack rings a counting problem in the master problem.  Compared
to~\eqref{prob:pricing1}, the resulting sub-problems do not contain the
recursive structure of \rcpp any more.  This balances the complexity between the
master problem and sub-problems, which is crucial for the performance of a
column generation algorithm.

\section{Column Generation Method for Solving the \rcpp}
\label{section:cgmethod}

This section presents a column generation based method for solving
formulation~\eqref{eq:pdwmodel}.  The key ingredients of our method are the
enumeration of valid circular patterns, the generation of rectangular patterns,
and a dynamic verification for circular pattern candidates during the column
generation algorithm.
Each of these fundamental components of the column generation based method are
necessary for addressing the complexity of solving the pricing
problem~\eqref{eq:pdwmodel:pricing} and the difficulty in verifying whether a
circular pattern candidate is packable. Since it may not be possible to compute
all possible circular and rectangular patterns, the dynamic verification
provides the capability to only check patterns that may be part of an optimal
solution.
Additionally, a key feature of our algorithm is that it is always capable of
computing valid primal and dual bounds even though not all rectangular and
circular patterns have been found.

\subsection{Enumeration of Circular Patterns}
\label{section:enumeration}

Formulation~\eqref{eq:pdwmodel} contains one variable for each circular pattern
in $\cp$. This set is, in general, of exponential size. We present a column
enumeration algorithm to compute all relevant circular patterns
for~\eqref{eq:pdwmodel}. The main step of the algorithm is to verify whether a
given tuple $(t,P)\in\typeset\times\Z_{+}^\ntypes$ is in the set $\cp$ or not. A
tuple can be checked by solving the following nonlinear nonconvex verification
problem:
\begin{subequations}\label{eq:verifynlp}
\begin{align}
  \norm{\mvec{x_i}{y_i} - \mvec{x_j}{y_j}}_2 \ge \rext_i + \rext_j & \fa i,j \in C: i < j \label{eq:verifynlp:cons1} \\
  \norm{\mvec{x_i}{y_i}}_2 \le \rint_t - \rext_i & \fa i \in C \label{eq:verifynlp:cons2} \\
  x_i, y_i \in \R & \fa i \in C \label{eq:verifynlp:cons3}
\end{align}
\end{subequations}
Here $C := \{1, \ldots, \sum_{i}P_i\}$ is the index set of individual circles,
and $\rext_i$ the corresponding external radius of a circle $i \in
C$. Model~\eqref{eq:verifynlp} checks whether all circles can be placed in a
non-overlapping way into a ring of
type~$t\in\typeset$. Constraint~\eqref{eq:verifynlp:cons1} ensures that no two
circles overlap, and~\eqref{eq:verifynlp:cons2} guarantees that all circles are
placed inside a ring of type $t$.

The non-overlapping conditions~\eqref{eq:verifynlp:cons1} are nonconvex and
make~\eqref{eq:verifynlp} computationally difficult to solve. Due to the
positioning of circles,~\eqref{eq:verifynlp} contains a lot of symmetry. The
rotation of every solution by 180\degree{} leads to another equivalent solution.
Typically, this kind of symmetry is difficult to address within a global \NLP
solver. Branching on some continuous variables is likely to have no impact on
the dual bound. One way to overcome this problem is to break some symmetry of
the problem by ordering circles of the same type in the $x$-coordinate
non-decreasingly. We achieve this by adding
\begin{equation}\label{eq:orderingx}
  x_i \le x_j \fa i < j: \rext_i = \rext_j
\end{equation}
to~\eqref{eq:verifynlp}. Additionally, we add an auxiliary objective function
$\min_{i \in C} x_i$ to~\eqref{eq:verifynlp}. From our computational experiments
we have seen that adding~\eqref{eq:orderingx} to~\eqref{eq:verifynlp} makes it
easier for the \NLP solver to prove infeasibility.

As already seen in Section~\ref{section:circularpatterns}, there is a dominance
relation between circular patterns, meaning that in any optimal solution of
\rcpp, dominated patterns can be replaced by non-dominated
ones. Definition~\ref{def:dominance} formalizes this notion of a dominance
relation between circular patterns.
\begin{definition}\normalfont\label{def:dominance}
  A circular pattern $(t,P)\in \cp$ \emph{dominates} $(t,P') \in \cp$ if and
  only if $P' <_{lex} P$, where $<_{lex}$ denotes the standard lexicographical
  order of vectors.
\end{definition}
Let
\begin{equation*}
  \cpdom := \left\{ (t,P) \in \cp \mid \nexists (t,P') \in \cp: (t,P') \text{ dominates } (t,P) \right\}
\end{equation*}
be the set of \emph{non-dominated circular patterns}. The set $\cpdom$ might be
much smaller than $\cp$, but is, in general, still of exponential size. Using
$\cpdom$ in~\eqref{eq:pdwmodel}, instead of the larger set $\cp$, results in
fewer variables.

Finally, we present the procedure \emph{EnumeratePatterns} in order to compute
$\cpdom$. Algorithm~\ref{alg:enum} considers all possible
$(t,P) \in \typeset\times\Z_{+}^\ntypes$ and checks whether $(t,P)$ is a
circular pattern by solving~\eqref{eq:verifynlp}. The algorithm exploits the
dominance relation between circular patterns to reduce the number of \NLP solves
and filter dominated patterns.
In the following, we discuss the different steps of Algorithm~\ref{alg:enum} in
more detail. For simplicity, we define the set $[x] := \{0,\ldots,x\}$ for an
integer number $x \in \Z_+$, and call a candidate
pattern~$(t,P) \in \typeset \times \Z_{+}^\ntypes$ \emph{infeasible} if
$(t,P) \notin \cp$ and feasible otherwise.
\begin{algorithm}[h]
  \caption{EnumeratePatterns}
  \Input{internal and external radii $\rint$ and $\rext$, demands $\demand$}
  \Output{$\cpfeas \subseteq \cp$, unverified candidates $\cpunknown$}
  \label{alg:enum}
  $\cpfeas := \emptyset$, $\cpinfeas := \emptyset$, $\cpunknown := \emptyset$ \\
  \For{$t \in \typeset$}
  {
    \For{ $P \in [\demand_1] \times \ldots \times [\demand_\ntypes]$ \label{alg:enum:select} }
    {
      \If{ $\exists (P',t) \in \cpinfeas : P' <_{lex} P$ or $\exists (P',t) \in \cpfeas : P' >_{lex} P$ }
      {
        \textbf{continue} \label{alg:enum:dom}
      }
      $status :=$ solve verification \NLP~\eqref{eq:verifynlp} \label{alg:enum:solve}\\
      \If{ $status = $ "feasible" }
      {
        $\cpfeas := \cpfeas \cup {(P,t)}$
      }
      \If{ $status = $ "infeasible" }
      {
        $\cpinfeas := \cpinfeas \cup {(P,t)}$
      }
      \If{ $status = $ "time limit" or "memory limit" }
      {
        $\cpunknown := \cpunknown \cup {(P,t)}$
      }
    }
  }
  filter all dominated patterns from $\cpfeas$ and $\cpunknown$ \label{alg:enum:filter} \\
  \Return $(\cpfeas,\cpunknown)$
\end{algorithm}

The algorithm maintains three sets~$\cpfeas$, $\cpinfeas$, $\cpunknown$,
initialized to the empty set.
In Line~\ref{alg:enum:select}, we iterate through all possible pattern
candidates $(t,P)$ for a fixed $t \in \typeset$.
In Line~\ref{alg:enum:dom}, we check whether $P$ dominates a circular
pattern already verified as infeasible and whether $P$ is
dominated by a circular pattern already verified as feasible. In both cases, $P$ can be skipped.
Otherwise, in Line~\ref{alg:enum:solve}, we solve a nonconvex verification
\NLP~\eqref{eq:verifynlp}, which is the bottleneck of Algorithm~\ref{alg:enum}.
Roughly speaking, this \NLP is easy to solve for the case where $P$, selected in
Line~\ref{alg:enum:select}, is component-wise too small or too large. In the
first case, finding a feasible solution is easy due to a small number of
circles. In the other case, we can conclude that the circles of the candidate
$P$ cannot be packed due to limited volume of the surrounding ring. The
computationally expensive verification \NLPs lie between these two extreme
cases.
Because of the two reversed dominance checks, it is in general unclear in which order the~$P$ in
Line~\ref{alg:enum:select} should be enumerated. On the one hand, it can be beneficial to start with
candidates that contain many circles. In case of a feasible candidate, many
other candidates can be discarded because of the dominance relation between
circular patterns. On the other hand, an infeasible candidate that dominates
another infeasible candidate can incur a redundant, difficult
\NLP~\eqref{eq:verifynlp} solve.

Algorithm~\ref{alg:enum} returns two sets of circular patterns. The first set
contains all feasible circular patterns that could be successfully verified and
is denoted by $\cpfeas$. The second set, $\cpunknown$, contains all candidates
that could not be verified because of working limits, e.g., a time or memory
limit on the solve of~\eqref{eq:verifynlp}. At the end, each non-dominated circular pattern
is either in the set $\cpfeas$ or $\cpunknown$, which means that
\begin{align*}
  \cpdom \subseteq \cpfeas \cup \cpunknown
\end{align*}
holds.

Ideally, we have verified all candidates, i.e., $\cpunknown = \emptyset$. However, since
there are exponentially many candidates to check and each candidate requires to
solve~\eqref{eq:verifynlp}, it might not be possible to compute $\cpdom$
in a reasonable amount of time.

Nevertheless, even if the set $\cpunknown$ is non-empty, we are able to compute valid
dual and primal bounds for \rcpp.
A valid primal bound is given when solving~\eqref{eq:pdwmodel} after restricting
the set of circular patterns to $\cpfeas$. This is becasuse this restriction
ensures that we only use packable patterns.
Using $\cpfeas \cup \cpunknown$ in~\eqref{eq:pdwmodel} yields a valid relaxation
of \rcpp because we use at least all patterns in $\cpdom$ and maybe some
circular patterns that are not packable.

\subsection{Computation of Valid Dual Bounds}
\label{section:validbounds}

The Pricing Problem~\eqref{eq:pdwmodel:pricing} is a classical \cpp---an
$\mathcal{NP}$-hard nonlinear optimization problem, which is in practice very
hard to solve. State-of-the-art \MINLP solvers might fail to prove global
optimality for this type of problems. Nevertheless, the following theorem shows
how to use a dual bound for~\eqref{eq:pdwmodel:pricing} to compute a valid dual
bound for \rcpp.

\begin{theorem}[Farley~\cite{Farley1990}, Vance et al.~\cite{Vance1994}]\label{theorem:farley}\normalfont
  Let $\nu_{RMP}$ be the optimum of the restricted master problem of
  $\PDW(\rp')$, $\nu_{Pricing}$ be the optimal value for the Pricing
  Problem~\eqref{eq:pdwmodel:pricing}, and $OPT$ be the optimal solution value
  of \rcpp. Then the inequality
  \begin{align*}
    \ceil{\frac{\nu_{RMP}}{1 - z_{Pricing}}} \le OPT
  \end{align*}
  holds for all valid dual bounds $z_{Pricing} \le \nu_{Pricing}$.
\end{theorem}

In our computational experiments we have seen that this bound can
indeed be used to obtain good quality bounds in cases where the circle packing
pricing problem~\eqref{eq:pdwmodel:pricingformulation} could not be solved to
optimality. Note that the dual bounds of Theorem~\ref{theorem:farley} depend on
the quality of the dual bounds of pricing problems. Any improvement on the dual
bound for the \cpp automatically translates to better dual bounds for the \rcpp.

\subsection{Price-and-Verify Algorithm}
\label{section:price-and-verify}

The price-and-verify algorithm is a column generation based algorithm that
incorporates ideas from Sections~\ref{section:pdw}, \ref{section:enumeration},
and~\ref{section:validbounds}. This is summarized in
Algorithm~\ref{alg:priceandverify} and consists of three main steps:
\begin{enumerate}
  \item An initial enumeration of circular patterns.
  \item Generation of rectangular patterns with negative reduced cost.
  \item Verification of circular pattern candidates during the pricing loop.
\end{enumerate}
First, in Line~\ref{alg:priceandverify:enumerate}, we use
Algorithm~\ref{alg:enum} to compute the set of non-dominated circular patterns
$\cpdom$. Its computational cost depends on the number of different ring types
$\ntypes$, the external radii $\rext$, and the internal radii $\rint$.
Algorithm~\ref{alg:enum} returns two sets $\cpfeas$ and $\cpunknown$ with the
properties
\begin{align*}
  \cpfeas & \subseteq \cp \text{ and} \\ \cpdom & \subseteq \cpfeas \cup
  \cpunknown.
\end{align*}

A feature of Algorithm~\ref{alg:priceandverify} is the ability to compute valid
dual and primal bounds simultaneously. This is achieved by computing the primal
and dual bounds while dynamically verifying circular pattern candidates in
$\cpunknown$ during the solving process. The simultaneous computation of bounds
is an improvement over the methods previously discussed. In the previous
section, the proposed methods separately compute valid primal and dual bound for
\rcpp by using $\cpfeas$ and $\cpfeas \cup \cpunknown$ respectively.

Algorithm~\ref{alg:priceandverify} uses $\cpfeas \cup \cpunknown$ as the initial
set of circular patterns in Formulation~\eqref{eq:pdwmodel}. This ensures
that at least all packable circular patterns are considered, which is necessary
for proving a valid dual bound for \rcpp. In general, many pattern candidates in
$\cpunknown$ are not packable and need to be discarded to prove global
optimality.  Algorithm~\ref{alg:priceandverify} dynamically verifies $C \in
\cpunknown$ and fixes the corresponding variable $z_C$ to zero if $C$ is not
packable. By discarding non-packable pattern candidates, the verification step
does not only improve the quality of the dual bound but also ensures that any
integer feasible solution $\bar z$ is feasible for the \rcpp, i.e., $z_C = 0$
for all $C \in \cpunknown$.

The key idea of the dynamic verification is to only consider candidate patterns
in $\cpunknown$ that have a nonzero \LP solution value. To be more precise, let
$z^*$ be the optimal \LP solution of the restricted master problem after no more
rectangular patterns with negative reduced cost could be found in the pricing
loop, see
Line~\ref{alg:priceandverify:pricingloop}. Algorithm~\ref{alg:priceandverify}
solves the \LP relaxation of the master problem~\eqref{eq:pdwmodel} to
optimality if $z^*_C = 0$ holds for all $C \in \cpunknown$. Otherwise, there
exists at least one $C \in \cpunknown$ with $z_C^* > 0$. In order to verify $C$,
we solve~\eqref{eq:verifynlp} in Line~\ref{alg:priceandverify:checkpattern} with
larger working limits than we have used in the initial enumeration step. There
are three possible outcomes. The candidate pattern $C$
\begin{itemize}
\item \textbf{is packable:} We remove $C$ from $\cpunknown$, add it to
  $\cpfeas$, and continue with the next pattern candidate $C \in \cpunknown$
    that has a nonzero solution value in $z^*$.

\item \textbf{is not packable:} We remove $C$ from $\cpunknown$ and fix $z_C$ to
  zero, which cuts off the \LP solution $z^*$. Resolving the \LP leads to a
  different dual solution that might allow us to find new rectangular patterns
    with negative reduced cost. In this case, Algorithm~\ref{alg:priceandverify}
    goes back to Line~\ref{alg:priceandverify:pricingloop} and continues with
    pricing.

\item \textbf{could not be verified:} Due to working limits, it may not be
  possible to verify $C$. In this case, we label that $C$ has been tested, i.e.,
    $\Psi(C)=1$, and continue with the next candidate $C' \in \cpunknown$ that
    has not been labeled yet, i.e., $\Psi(C')= 0$, and $z_{C'}^* >0 $. If $C'$
    can be verified to be not packable, we might get a different \LP solution
    with $z_C^* = 0$. This would allow us to solve the \LP relaxation of the
    master problem to optimality even though we could not verify $C$.

  In Line~\ref{alg:priceandverify:invalidcheck} we check whether there is still
    a candidate $C' \in \cpunknown$ with $z_{C'}^* > 0$ left. The candidates
    where $z_{C'}^* > 0$ have already been tested, so $\Psi(C') = 1$, but were
    unable to be verified within the working limits.  We fix all variables
    $z_{C'}$ to zero for all $C' \in \cpunknown$.  Since unverified patterns
    have been fixed to zero, it is no longer possible to compute a valid dual
    bound for \rcpp.
  The remaining solution process can be seen as solving a restricted version of
    \rcpp whose solution $\bar z$ is feasible for the original problem. However,
    it might be the case that $\bar z$ is still optimal, even though
    unverified candidates from $\cpunknown$ are fixed to zero. It is possible to
    verify the circular patterns a posteriori to determine whether an optimal
    solution has been found.
\end{itemize}

The advantage of dynamically verifying circular pattern candidates during the
pricing loop is that small working limits can be used in
Algorithm~\ref{alg:enum} in order to identify packable patterns quickly and then
focus, with larger working limits, on the patterns that are used in the \LP
solution of the restricted master problem.

After applying Algorithm~\ref{alg:priceandverify} we have solved the \LP
relaxation of~\eqref{eq:pdwmodel} in the root node of the branch-and-bound tree.
If there exist any integer variables with a fractional solution value, then
branching must be performed.  Due to the complexity of the pricing problems, we
observed in our experiments that solving~\eqref{eq:pdwmodel:pricing} to global
optimality is only possible for simple problems that can be solved in the root
node---without requiring branching. The proposed branching strategy does not
greatly reduce the complexity of the pricing problems at each node of the tree.
As a result, performing pricing in each node is very time consuming. Thus, to
improve the computational performance, pricing is only performed in the root
node. Afterwards the \rcpp is solved for the set of rectangular patterns that
have been found so far. The optimal solution of this restricted problem is
feasible for the original \rcpp.  Our results show that applying this strategy
allows us to find good quality solutions for difficult problems.

\begin{algorithm}[t]
  \DontPrintSemicolon
  \caption{Price-and-verify}
  \label{alg:priceandverify}
  \Input{internal and external radii $\rint$ and $\rext$, demands $\demand$}
  \Output{\LP solution $z^*$ of the master problem or $\emptyset$}
  $(\cpfeas,\cpunknown) := $ EnumeratePatterns($\rint, \rext, \demand$) \label{alg:priceandverify:enumerate} \\
  $\Psi_C := 0$ for all $C \in \cpunknown$ \\
  \While{ $\exists R \in \rp:$ $\redcosts_R < 0$ \label{alg:priceandverify:pricingloop} }
  {
    $\rp := \rp \cup \{R\}$ \tcp*{pricing loop}
  }
  $z^* := $ solve $\LP(RMP)$ \label{alg:priceandverify:solvelp} \\
  \While{ $\exists C \in \cpunknown: z_C^* > 0 \land \Psi_C = 0$ \label{alg:priceandverify:checkpattern} }
  {
    $status :=$ solve verification \NLP~\eqref{eq:verifynlp} \label{alg:priceandverify:verify} \tcp*{verification step}
    $\Psi_C := 1$ \\
    \If{ $status = $ "feasible'' }
    {
      $\cpunknown := \cpunknown \backslash \{C\}$ \\
      $\cpfeas := \cpfeas \cup \{C\}$ \\
    }
    \If{ $status = $ "infeasible''}
    {
      $\cpunknown := \cpunknown \backslash \{C\}$ \\
      fix $z_C := 0$ \tcp*{fixing cuts of $z^*$}
      \goto{} \ref{alg:priceandverify:pricingloop} \tcp*{enter pricing loop again}
    }
  }
  \If{ $\exists C \in \cpunknown: z_C^* > 0$ \label{alg:priceandverify:invalidcheck} }
  {
    fix $z_C := 0$ for all $C \in \cpunknown$ \\
    \Return $\emptyset$ \tcp*{\LP solution is not valid}
  }
  \Return $z^*$
\end{algorithm}
 
\section{Computational Experiments}\label{section:experiments}

In this section, we investigate the performance and the quality of the dual and
primal bounds obtained by our method and analyze how they
relate to specific properties of an instance. The algorithm presented in
Section~\ref{subsection:improvedreformulation} is implemented in the \MINLP
solver \scip~\cite{SCIP}. We refer to~\cite{Achterberg2007a,Vigerske2013} for an
overview of the general solving algorithm and \MINLP features of \scip.

\subsection{Implementation}

We extended \scip with the addition of two plug-ins: one pricing plug-in for
solving the LP relaxation of Formulation~\eqref{eq:pdwmodel} and one constraint
handler plug-in to apply the dynamic verification of circular patterns during
Algorithm~\ref{alg:priceandverify}.
Algorithm~\ref{alg:enum} is executed immediately before the solving process for
the \rcpp commences.
To accelerate the verification of circular pattern candidates, we use a
simple greedy heuristic to check whether a given candidate $(t,P)$ is a circular
pattern, i.e., if $(t,P) \in \cp$. The heuristic iteratively packs circles to
the left-most, and then lowest possible position in a ring of type
$t\in\typeset$. If the heuristic fails to verify a candidate, we
solve~\eqref{eq:verifynlp} until a feasible solution has been found, or it has
been proven to be infeasible.
The same heuristic is used for finding a rectangular pattern with negative
reduced cost during Algorithm~\ref{alg:priceandverify}. If the heuristic fails
to find such a pattern, we directly solve~\eqref{eq:pdwmodel:pricing} to global
optimality.
The implementation is publically available in source code as part of the SCIP
Optimization Suite and can be downloaded at \url{https://scip.zib.de/}.

\subsection{Experimental Setup}

\newcommand{\realsetninstances}{9}
\newcommand{\randsetninstances}{800}

We conducted three main experiments. In the first experiment we characterize
instances for which Algorithm~\ref{alg:enum} finds all elements of the set of
non-dominated circular patterns $\cpdom \subseteq \cpfeas$. The second
experiment answers the question whether our proposed method is able to solve
instances to global optimality and characterizes these instances by their
structural properties. Because our method can also be used as a primal
heuristic, in the last experiment we compare it with the \grasp heuristic of
Pedroso et al.~\cite{Pedroso2016}.

In the enumeration experiment we apply Algorithm~\ref{alg:enum} on each instance
and check whether $\cpdom$ could be computed in two hours. For very difficult
problems it might happen that we spend the whole time limit in solving a single
\NLP.

For our second experiment, we use our method to compute valid dual and primal
bounds for \rcpp with a total time limit of two hours. In contrast to the first
experiment, we enforce a time limit of $10s$ for each \NLP~\eqref{eq:verifynlp}
in Algorithm~\ref{alg:enum}. After a time budget of $1200s$, we stop
solving~\eqref{eq:verifynlp} and only use the greedy heuristic to verify
circular pattern candidates. A pattern is added to $\cpfeas$ if it can be
verified. Otherwise, we add a candidate to $\cpunknown$ and process with the
next candidate pattern. During the pricing loop we then use a larger time limit
of $120s$ to verify a pattern in $\cpunknown$ that has a nonzero value in the LP
relaxation solution.  Again, we stop solving \NLPs after $2400s$ were spent on
verifying pattern candidates.

During Algorithm~\ref{alg:priceandverify}, we use a time limit of
$300s$ to solve~\eqref{eq:pdwmodel:pricing}. If we fail to solve a pricing
problem to optimality and no improving column could be found, we stop solving
any further pricing problems, obtain a valid dual bound by applying
Theorem~\ref{theorem:farley}, and continue to solve the restricted master
problem for the current set of rectangular patterns. In our experiments, this
only occurs at the root node.

Finally, in our third experiment we compare the obtained primal bounds of our
method with those obtained from the \grasp heuristic. Both algorithms run with a
time limit of three hours on each instance. For this experiment the
specifications from the second experiment are used for out method.
We use the Python implementation of \grasp from~\cite{Pedroso2016}. Note that
\scip is written in the programming language C, in which an implementation of
\grasp would be much faster. However, in our primal bound experiment we only
compare the quality of obtained solutions and it can be observed that \grasp
finds its best solutions in the first few minutes.

\paragraph{Test Sets.}

We consider two different test sets for our experiments. The first one contains
\realsetninstances{} real-world instances from the tube industry, which were
used in~\cite{Pedroso2016}, and is in the following called \realset test set.
Because this test set is too limited for a detailed computational study, we
created a second test set containing \randsetninstances{} instances, the
\randset test set. The purpose of this set is to show how the number of
different ring types $\ntypes$, the maximum ratio of external radii
${\max_t \rint_t} / {\min_t \rext_t}$, and the ratio between rectangle size and
the maximum volume of a ring
${\max\{\width,\height\}} / {\max_t \pi (\rext_t)^2}$ influence the performance
of our method.

The name of each instance reads
\texttt{i<$\ntypes$>\_<$\alpha$>\_<$\beta$>\_<$\gamma$>.rpa} where
\begin{itemize}
\item $\ntypes \in \{3,4,5,10\}$ is the number of ring types,
\item $\frac{\max_t \rint_t}{\min_t \rext_t} = :\alpha \in \{2.0,2.3,\ldots,4.7\}$
  is the maximum external radii ratio,
\item $\frac{\max\{\width,\height\}}{\max_t \rext_t} =: \beta \in
  \{2.0,2.3,\ldots,4.7\}$ is the rectangle size to external radius ratio, and
\item $\frac{\width \height}{ \max_t \pi (\rext_t)^2} \left[0.8 \gamma, 1.2 \gamma
  \right]$ for $\gamma \in \{5,10\}$ is the demand interval for each type $t$.
\end{itemize}
The demand of a type $t \in \typeset$ is randomly chosen from the corresponding
demand interval. The size of this interval is anti-proportional to the external
radius, i.e., types with large external radius appear less often. For all
instances in \randset we fixed the width $\width$ and height $\height$ of the
rectangles to~10. We created one instance for each 4-tuple, giving 800~instances in total.
All instances of the \randset and \realset test set are publically available at
  \url{https://github.com/mueldgog/RecursiveCirclePacking}.

\paragraph{Hardware and Software.}

The experiments were performed on a cluster of 64bit Intel(R) Xeon(R) CPU
E5-2660 v3 2.6\,GHz with 12\,MB cache and 48\,GB main memory.
In order to safeguard against a potential mutual slowdown of parallel processes,
we ran only one job per node at a time.
We used \scip version~5.0 with \cplex~12.7.1.0 as LP~solver~\cite{Cplex},
\cppad~20140000.1~\cite{CppAD}, and \ipopt~3.12.5 with
\mumps~4.10.0~\cite{Mumps} as \NLP solver~\cite{WachterBiegler2006,Ipopt}.

\subsection{Computational Results}

In the following, we discuss the results for the three described experiments in
detail. Instance-wise results of all experiments can be found in
Table in the electronic supplement.

\paragraph{Enumeration Experiments.}

Figure~\ref{fig:enumplot} shows the computing time required to compute all
non-dominated circular patterns $\cpdom$ for the instances of the \randset test
set. Each line corresponds to the subset of instances with identical number of
ring types $\ntypes$. Each point on a line is computed as the shifted
geometric mean (with a shift of 1.0) over all instances that have the same value
${\max_t \rint_t}/{\min_t \rext_t}$. We expect that the number of circular
patterns increases when increasing the ratio between the ring with largest inner
and the ring with smallest outer radius.

The first observation is that the time to compute all circular patterns
increases when $\ntypes$ increases. For example, for $\ntypes = 10$ we need
about $4-10$ times longer to enumerate all patterns than for $\ntypes = 5$.
Also, all lines in Figure~\ref{fig:enumplot} approximately increase
exponentially in ${\max_t \rint_t}/{\min_t \rext_t}$. A larger ratio
implies that each verification \NLPs~\eqref{eq:verifynlp} has a larger number of
circles that fit into a ring of inner radius $\max_t \rint_t$. These difficult
\NLPs and the larger number of possible circular patterns provide an
explanation for the exponential increase of each line in
Figure~\ref{fig:enumplot}.

\begin{figure}[ht]
  \centering
  \begin{tikzpicture}
    \begin{axis}[
      xlabel = $\frac{\max_t \rint_t}{\min_t \rext_t}$,
      ylabel = time in seconds,
      xmin   = 2.0,
      xmax   = 3.5,
      ymax   = 7200,
      ymode  = log,
      legend cell align=left,
      legend pos=south east,
      legend style={draw=none}
      ]
      \addplot table {
2.000000 0.544380
2.300000 1.073888
2.600000 17.807337
2.900000 42.885077
3.200000 386.784900
3.500000 5099.140768
3.800000 7200.000000
4.100000 7200.000000
4.300000 7200.000000
4.400000 7200.000000
4.500000 7200.000000
4.700000 7200.000000
};
\addlegendentry{$\ntypes$ = 3};

\addplot table {
2.000000 0.773332
2.300000 1.381454
2.600000 13.151805
2.900000 27.903659
3.200000 794.094392
3.500000 5539.757300
3.800000 7200.000000
4.100000 7200.000000
4.300000 7200.000000
4.400000 7200.000000
4.500000 7200.000000
4.700000 7200.000000
};
\addlegendentry{$\ntypes$ = 4};

\addplot table {
2.000000 0.891406
2.300000 3.066368
2.600000 4.569154
2.900000 172.872413
3.200000 958.085502
3.500000 6556.494959
3.800000 7200.000000
4.100000 7200.000000
4.300000 7200.000000
4.400000 7200.000000
4.500000 7200.000000
4.700000 7200.000000
};
\addlegendentry{$\ntypes$ = 5};

\addplot table {
2.000000 3.289793
2.300000 9.119621
2.600000 114.896597
2.900000 240.303821
3.200000 3310.817057
3.500000 7200.000000
3.800000 7200.000000
4.100000 7200.000000
4.300000 7200.000000
4.400000 7200.000000
4.500000 7200.000000
4.700000 7200.000000
};
\addlegendentry{$\ntypes$ = 10};

     \end{axis}
  \end{tikzpicture}
  \caption{Plot showing the average time to enumerate all circular patterns for
    different number of ring types $\ntypes \in \{3,4,5,10\}$ within a two hours
    time limit.}
  \label{fig:enumplot}
\end{figure}
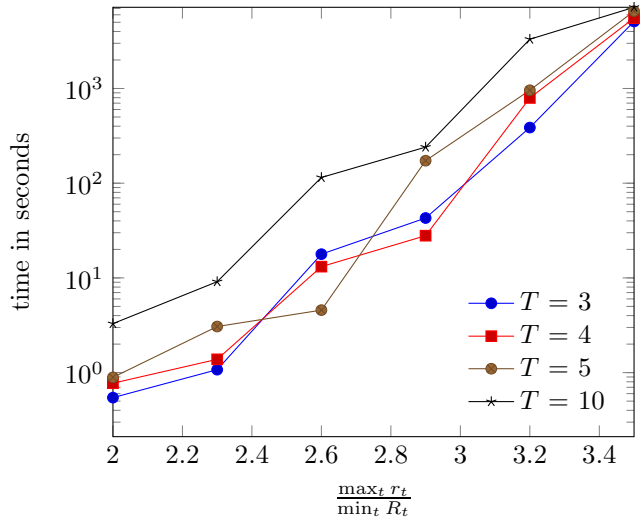

Table~\ref{res:enumtable} shows in more detail for how many instances we could
compute $\cpdom$ successfully and how much time it took on average. The rows
partition the instances according to the number of ring types and the columns
group these instances by their ${\max_t \rint_t} / {\min_t \rext_t}$
values. First, we see that for $392$ of the $800$ instances of the \randset test
set we could successfully enumerate all circular patterns, in $19.8$ seconds on
average. We see that for less and less instances the enumeration step is
successful as ${\max_t \rint_t} / {\min_t\rext_t}$ increases. For none of the
\randset instances with ${\max_t \rint_t} / {\min_t \rext_t}$
strictly larger than $3.5$ could $\cpdom$ be successfully
enumerated. Many verification \NLPs for these instances are too difficult to
solve and often consume the whole time limit of two hours, even for the case of
three different ring types.
However, for the 480~instances with ${\max_t \rint_t} / {\min_t \rext_t} \le 3.5$
Algorithm~\ref{alg:enum} managed to compute on $81.6$\% (392) of these instances all
non-dominated circular patterns.

\begin{table}[ht]
  \caption{Aggregated results for enumerating all non-dominated circular
    patterns. Each of the six columns reports the results for $80$ instances of the
      \randset test set. The $320$ instances with ${\max_t \rint_t} / {\min_t\rext_t} > 3.5$ are
      not shown because none of them could be successfully enumerated within the time limit.}
  \label{res:enumtable}
  \begin{tabular*}{\textwidth}{@{\extracolsep{\fill}}l|rr|rr|rr|rr|rr|rr}
    \tabledefline{n}{number of instances for which $\cpdom$ could be computed}
    \tabledefline{time}{time in seconds}
    \toprule
    $\nu = \frac{\max_t \rint_t}{\min_t \rext_t}$ & \multicolumn{2}{c|}{$\nu = 2.0$}
    & \multicolumn{2}{c|}{$\nu = 2.3$}
    & \multicolumn{2}{c|}{$\nu = 2.6$}
    & \multicolumn{2}{c|}{$\nu = 2.9$}
    & \multicolumn{2}{c|}{$\nu = 3.2$}
    & \multicolumn{2}{c}{$\nu = 3.5$}\\
    & n & time & n & time & n & time & n & time & n & time & n & time \\
    \midrule \midrule
    $\ntypes = 3$  & 20 & 0.5 & 20 & 1.0 & 20 & 8.1 & 20 & 35.2 & 17 & 227.5 & 12 & 4050.7 \\
$\ntypes = 4$  & 20 & 0.7 & 20 & 1.3 & 20 & 8.2 & 20 & 24.2 & 15 & 374.5 & 9 & 4020.3 \\
$\ntypes = 5$  & 20 & 0.9 & 20 & 2.6 & 19 & 4.3 & 19 & 117.9 & 16 & 571.4 & 4 & 4507.7 \\
$\ntypes = 10$ & 20 & 3.0 & 20 & 8.0 & 16 & 28.9 & 17 & 117.3 & 8 & 1022.6 & 0 & -- \\ \midrule
all            & 80 & 1.1 & 80 & 2.5 & 75 & 9.3 & 76 & 56.8 & 56 & 419.4 & 25 & 4109.4 \\
     \bottomrule
  \end{tabular*}
\end{table}

\paragraph{Exact Price-and-Verify.}

In our second experiment, we analyze the primal and dual bounds that have been
computed by our column generation method. Figure~\ref{fig:piechart} shows the
achieved optimality gaps, i.e., $(\text{primal} - \text{dual})/{\text{dual}},$
for all instances of the \randset test set. We solve $35.9$\% of the instances
to global optimality and get gaps between $0$-$25$\% for $5.6$\% of the
instances. For about $46.3$\% of the instances we achieve gaps between
$25$-$100\%$, and only for a single instance the gap is larger than $100$\%.

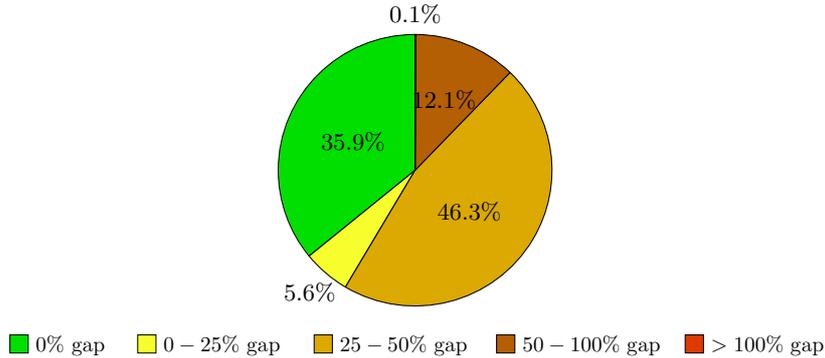
\begin{figure}[ht]

  \begin{center}
    \begin{tikzpicture}
      [
      pie chart,
      slice type={A}{c1},
      slice type={B}{c2},
      slice type={C}{c3},
      slice type={D}{c4},
      slice type={E}{c5},
      pie values/.style={font={\small}},
      scale=1.8
      ]
      \pie[values of E/.style={pos=1.15},values of B/.style={pos=1.18}]{}{35.9/A,5.6/B,46.3/C,12.1/D,0.1/E};
    \end{tikzpicture}

    \scalebox{0.80}{
      \begin{tikzpicture}
        \draw[fill=c1] (1,0) rectangle (1.3,0.3) node[right, yshift=-5]{$0\%$ gap};
        \draw[fill=c2] (3.1,0) rectangle (3.4,0.3) node[right, yshift=-5]{$0 - 25\%$ gap};
        \draw[fill=c3] (6,0) rectangle (6.3,0.3) node[right, yshift=-5]{$25 - 50\%$ gap};
        \draw[fill=c4] (9,0) rectangle (9.3,0.3) node[right, yshift=-5]{$50 - 100\%$ gap};
        \draw[fill=c5] (12.1,0) rectangle (12.4,0.3) node[right, yshift=-5]{$> 100\%$ gap};
      \end{tikzpicture}
    }
  \end{center}

  \caption{Optimality gaps for all \randset instances.}
  \label{fig:piechart}
\end{figure}

Table~\ref{res:gaptabletype} and Table~\ref{res:gaptablerratio} contain
aggregated results for the optimality gaps for the \randset test
set. Table~\ref{res:gaptabletype} shows that out of the~$287$ optimally solved
instances,~$80$ instances have three,~$72$ have four,~$70$ have five, and~$65$
have ten different ring types.
Most of the instances with a gap larger than~$50\%$ are from the subset of
instances with ten ring types. Only 33 of the~$98$ instances with a gap larger
than~$50\%$ have less than ten different ring types.
Figure~\ref{fig:bardiag} visualizes the results of Table~\ref{res:gaptabletype}
in a bar diagram. As expected, for an increasing number of ring types worse
optimality gaps are obtained.

\begin{table}[ht]
  \caption{Number of instances of the \randset test set according to final gap and number of ring types.}
  \label{res:gaptabletype}
  \begin{tabular*}{\textwidth}{@{\extracolsep{\fill}}lc||rrrrrr}
    \toprule
    & total & $0\%$ & $0-25\%$ & $25-50\%$ & $50-100\%$ & $>100\%$ \\
    \midrule
    \midrule
    $\ntypes= 3$      &        200 &         80 &         14 &         99 &          7 &          0 \\
$\ntypes= 4$      &        200 &         72 &         11 &        108 &          9 &          0 \\
$\ntypes= 5$      &        200 &         70 &          9 &        104 &         17 &          0 \\
$\ntypes= 10$     &        200 &         65 &         11 &         59 &         64 &          1 \\ \midrule
all               &        800 &        287 &         45 &        370 &         97 &          1 \\
     \bottomrule
  \end{tabular*}
\end{table}

\begin{figure}[ht]
  \begin{center}
    \scalebox{0.8}{
      \bardiagram{$|\mathcal{T}|$}{
  0/3/40.0/7.0/49.5/3.5/0.0,
  1/4/36.0/5.5/54.0/4.5/0.0,
  2/5/35.0/4.5/52.0/8.5/0.0,
  3/10/32.5/5.5/29.5/32.0/0.5}
     }
    \scalebox{0.80}{
      \begin{tikzpicture}
        \draw[fill=c1] (1,0) rectangle (1.3,0.3) node[right, yshift=-5]{$0\%$ gap};
        \draw[fill=c2] (3.1,0) rectangle (3.4,0.3) node[right, yshift=-5]{$0 - 25\%$ gap};
        \draw[fill=c3] (6,0) rectangle (6.3,0.3) node[right, yshift=-5]{$25 - 50\%$ gap};
        \draw[fill=c4] (9,0) rectangle (9.3,0.3) node[right, yshift=-5]{$50 - 100\%$ gap};
        \draw[fill=c5] (12.1,0) rectangle (12.4,0.3) node[right, yshift=-5]{$> 100\%$ gap};
      \end{tikzpicture}
    }
  \end{center}
  \caption{Optimality gaps for \randset instances split by number of different
    ring types.}
  \label{fig:bardiag}
\end{figure}
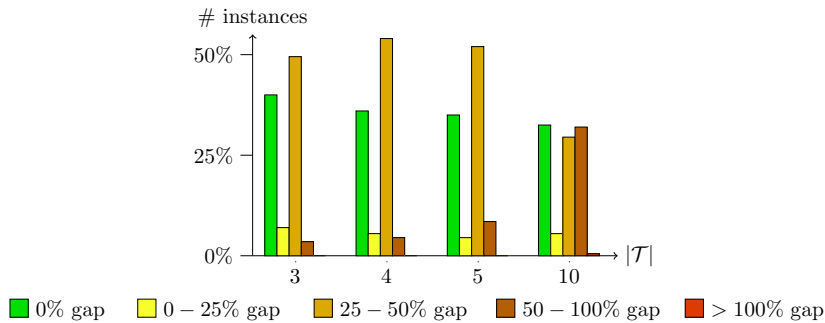

Each row in Table~\ref{res:gaptablerratio} corresponds to the set of instances
that have at least a certain value for ${\max_t \rint_t}/{\min_t \rext_t}$. For
example, the bottom-left corner value signifies that $60$ out of the
$160$~instances with ${\max_t \rint_t}/{\min_t \rext_t} \ge 4.4$ could be solved
to optimality. As for Table~\ref{res:gaptabletype}, we see that the gaps
increase with a larger ${\max_t \rint_t}/{\min_t \rext_t}$ ratio. This can be
explained by our previous observation, namely, the
difficulty of enumerating all circular patterns for these instances.
When $\cpunknown \neq \emptyset$, then it could contain packable and
unpackable patterns. A valid dual bound is given by the solution to the \LP
relaxation without any unverified patterns fixed to zero. If this solution
contains patterns from $\cpunknown$ that are unpackable, then the computed bound
will be lower than the optimal \LP bound if all circular patterns were verified.
Alternatively, after the \LP relaxation has been solved to optimality, all
unverified patterns are fixed to zero. If $\cpunknown$ contains packable
patterns that are necessary to express an optimal integer solution, then the best
primal bound will be greater than the optimal solution value to the \rcpp.
To summarize, our results show that the number of ring types $\ntypes$ and the
quotient~${\max_t \rint_t}/{\min_t \rext_t}$ are in many cases reliable
indicators for both the quality of primal and dual bounds and the computational
cost of our method. Instances that require branching are often
too difficult to solve to global optimality.

\begin{table}[ht]
  \caption{Number of instances of the \randset test set according to final gap and ${\max_t \rint_t}/{\min_t \rext_t}$.}
  \label{res:gaptablerratio}
  \begin{tabular*}{\textwidth}{@{\extracolsep{\fill}}lc||rrrrrr}
    \toprule
    & total & $0\%$ & $0-25\%$ & $25-50\%$ & $50-100\%$ & $>100\%$ \\
    \midrule
    \midrule
    ${\max_t \rint_t}/{\min_t \rext_t} \ge 2.0$                &        800 &        287 &         45 &        370 &         97 &          1 \\
\hphantom{${\max_t \rint_t}/{\min_t \rext_t}$}\;$\ge 2.6$  &        640 &        223 &         42 &        295 &         79 &          1 \\
\hphantom{${\max_t \rint_t}/{\min_t \rext_t}$}\;$\ge 3.2$  &        480 &        170 &         36 &        204 &         69 &          1 \\
\hphantom{${\max_t \rint_t}/{\min_t \rext_t}$}\;$\ge 3.8$  &        320 &        122 &         25 &        126 &         46 &          1 \\
\hphantom{${\max_t \rint_t}/{\min_t \rext_t}$}\;$\ge 4.4$  &        160 &         60 &         11 &         66 &         23 &          0 \\
     \bottomrule
  \end{tabular*}
\end{table}

Finally, we briefly report on the size of the branch-and-bound trees for the
primal and dual bound experiments. We consider a node to be explored once it has
been selected by \scip's node selection algorithm. Our method explored more than
one node for $367$ of the $800$ \randset instances and only $15$ of these
instances could be solved to global optimality. The maximum number of explored
nodes is $207$.

\paragraph{Primal Bound Experiments.}

In our final experiments, we consider our method as a pure
primal heuristic for \rcpp and compare it to the \grasp heuristic by Pedroso et
al.~\cite{Pedroso2016}, which is specifically designed to find dense packings
quickly. Table~\ref{res:detailedreal} contains detailed results for the
\realset, and Table~\ref{res:primalbounds} aggregated results for the \randset test
set. For analytical purposes, Table~\ref{res:detailedreal} also shows the
results for the dual bounds.

\begin{table}[ht]
  \caption{Detailed results for the comparison with \grasp on the \realset test set.}
  \label{res:detailedreal}
  \centering
  \begin{tabular*}{30em}{l@{\extracolsep{\fill}}rrrrrr}
    \toprule
    \multicolumn{4}{c}{Instance} & \multicolumn{2}{c}{price \& verify} & \multicolumn{1}{c}{\grasp}\\
    \cmidrule[0.5pt]{1-4} \cmidrule[0.5pt]{5-6} \cmidrule[0.5pt]{7-7}
    Name\quad\quad & \ntypes & $\frac{\max_t \rint_t}{\min_t \rext_t}$ & volume & dual & primal & primal \\
    \midrule
    \midrule
    s03i1 & 3  & 2.2 & 1  & 1 & 1 & 1\\
s03i2 & 3  & 2.2 & 5  & 9 & 10 & 10\\
s03i3 & 3  & 2.2 & 47 & 82 & 98 & 95\\
s05i1 & 5  & 5.1 & 1  & 1 & 1 & 1\\
s05i2 & 5  & 5.1 & 4  & 8 & 10 & 10\\
s05i3 & 5  & 5.1 & 36 & 70 & 94 & 98\\
s16i1 & 16 & 6.3 & 1  & 1 & 1 & 1\\
s16i2 & 16 & 6.3 & 5  & 8 & 10 & 10\\
s16i3 & 16 & 6.3 & 45 & 73 & 96 & 96\\
     \bottomrule
  \end{tabular*}
\end{table}

For the \realset test set, we are able to solve the three
instances~\texttt{s**i1}, that require only one rectangle in the optimal
solution. On all but two instances our method achieves the same primal
bound as \grasp. On \texttt{s03i3} we need three more rectangles and
on \texttt{s05i3} four rectangles less.
Due to the small value ${\max_t \rint_t}/{\min_t \rext_t} = 2.2$ of the
instances with three different ring types, the enumeration step takes only a
fraction of a second. The difficulty of these instances lies only in the
placement of rings into rectangles and not in the recursive structure of
\rcpp. Hence, it is not surprising that our method computes a worse solution than
\grasp on \texttt{s03i3} because we only use a simple left-most greedy heuristic
before solving~\eqref{eq:pdwmodel:pricing}.

Proving optimality for the instances of the \realset test set is difficult for
our method.  Only the three instances~\texttt{s**i1} with one rectangle in the
optimal packing, could be solved to gap zero.
Because the rectangles are very large relative to the radii of the rings, the
pricing problems contain many variables and constraints. As a result, the
pricing problems are too difficult to solve to optimality. As shown above, the
ratio ${\max_t \rint_t}/{\min_t \rext_t}$ of the instances with $\ntypes > 3$ is
too large to enumerate and verify all circular patterns in reasonable time.

It is worth to notice that the dual bounds that can be proved by our method are
much better than simple volume-based bounds. Except for the three instances that
require only one rectangle, our dual bounds are at least $1.6$ times larger than
the volume-based bounds.

Even though our method is not particularly designed for the characteristics of
the instances in the \realset test set, it still performs well as a primal
heuristic. On the instances with five ring types it could enumerate many
non-trivial circular patterns and uses them to find a better primal solution
than \grasp.

For the \randset test set, Table~\ref{res:primalbounds} shows the number of instances on which
our method performed better or worse than \grasp with respect to the achieved
primal bound. The second column reports the primal bound average relative to the
value of \grasp on all instances of \randset. For this we compute the shifted
geometric mean of all ratios between the obtained primal bounds of our method
and \grasp. The third column shows the total number of instances on which our
method found a better primal solution. The fourth column shows the improvement
relative to \grasp. The remaining three columns show statistics for the instances on
which our method performed worse.

We observe that our method outperforms \grasp on many instances.  On average,
the solutions are $2.7\%$ better than the ones from \grasp. In total, we find a
better primal bound on $356$ of the instances, and a worse bound on only $33$
instances. Our method could find for all instances a primal feasible solution
before hitting the time or memory limit. The average deterioration on the $33$
instances is~$5.1-7.2\%$ and the average improvement on the $356$ instances is
$5.6-6.9\%$.
Our method finds more often a better solution than \grasp for instances that
have a larger number of ring types. For ten ring types we find $123$ better
solutions, which is about twice as large as for three ring types.

There are two reasons why our method performs better than \grasp on many
instances.  The first is that \grasp considers and positions each ring
individually. A large demand vector $\demand$ results in a large number of
rings, which makes it difficult for \grasp to find good primal solutions. In
contrast, the number of rings and circles that need to be considered in our
pricing problems and in the enumeration is typically bounded by a number derived
from some volume arguments instead of the entries of the demand vector
$\demand$. Thus, scaling up $\demand$ typically does not have a large impact
when applying our column generation algorithm as a primal heuristic. The second
reason for the better performance is that for a given set of circular and
rectangular patterns, the master problem takes the best decisions of how to pack
rings recursively into each other. For instances where this combinatorial part
of the problem is difficult, we expect that our method performs better than
\grasp. Indeed, on instances with a large number of ring types our method
frequently finds better solutions than \grasp.

\begin{table}[ht]
  \caption{Aggregated results for primal bound comparisons against \grasp on the \randset test
    set.}
  \label{res:primalbounds}
  \centering
  \begin{tabular*}{35em}{@{\extracolsep{\fill}}lrrrrr}
    \tabledefline{all}{all instances}
    \tabledefline{better}{instances with a better primal bound than \grasp}
    \tabledefline{worse}{instances with a worse primal bound than \grasp}
    \tabledefline{\#}{number of instances}
    \tabledefline{\%}{average primal bound relative to average primal bound of \grasp}
    \toprule
    & \multicolumn{1}{c}{all} & \multicolumn{2}{c}{better} & \multicolumn{2}{c}{worse}\\ \cmidrule[0.5pt]{2-2}\cmidrule[0.5pt]{3-4}\cmidrule[0.5pt]{5-6}
    & \% & \# & \% & \# & \% \\
    \midrule
    \midrule
    $\ntypes$ = 3  &       98.3 &         61 &       94.3 &          7 &      106.3\\
$\ntypes$ = 4  &       97.3 &         78 &       93.1 &          9 &      105.5\\
$\ntypes$ = 5  &       97.4 &         94 &       94.4 &          9 &      107.2\\
$\ntypes$ = 10 &       96.1 &        123 &       93.6 &          8 &      105.1\\ \midrule
all            &       97.3 &        356 &       93.8 &         33 &      106.0\\
     \bottomrule
  \end{tabular*}
\end{table}

In summary, the experiments on synthetic and real-world instances shows that our
method solves small and medium-sized instances to global optimality.  This was
the case on $287$ of the randomly generated instances and for three of nine
real-world instances.  Due to the costly verification of circular patterns and
difficult sub-problems, our method is not able to solve larger instances to
global optimality, but still achieves good primal and dual bounds.  Compared to
the state-of-the-art heuristic for \rcpp, our method finds on $356$ of the
instances better, and only on $33$\% of the instances worse primal solutions.
 
\section{Conclusion}
\label{section:conclusion}

In this article, we have presented the first exact algorithm for solving
practically relevant instances for the extremely challenging recursive circle packing problem. Our method is based
on a Dantzig-Wolfe decomposition of a nonconvex \MINLP model. The key idea of
solving this decomposition via column generation is to break symmetry of the
problem by using circular and rectangular patterns. These patterns are used to
model all possible combinations of packing rings into other rings. As a result,
we were able to reduce the complexity of the sub-problems significantly and
shift the recursive part of \rcpp to a linear master problem.

In some sense, this reformulation can be interpreted as a reduction technique from \rcpp to \cpp. All
occurring sub-problems in the enumeration of circular patterns and the pricing
problems are classical \cpps and they constitute the major computational
bottleneck. Every primal or dual improvement for this problem class would
directly translate to better primal and dual bounds for \rcpp. Still, in order
to prove optimality it is usually necessary to solve the $\mathcal{NP}$-hard
$\cpp$ to optimality, which can fail for harder instances. However,
even for this case, the application of Theorem~\ref{theorem:farley} and the
pessimistic and optimistic enumeration of circular patterns guarantees valid
primal and dual bounds for \rcpp. The combination of column generation with
column enumeration could be of more general interest when using
decomposition techniques for \MINLPs that lead to hard nonconvex sub-problems.

An interesting extension of the presented method is its application
to problem with different container shapes and sizes. Since the master problem
of the decomposed model does not depend on the shape of the containers or packed
objects, it is possible to apply this model to any container shape. This would
only require a modification to the column generation subproblems to produce
packable patterns that can be added to the master problem. Thus, the presented
methods cover a wide range of packing problems.

Finally, our current proof-of-concept implementation could certainly be improved
further. To mention only one point, we currently use a rather simple greedy
heuristic in order to find quickly a feasible solution for the verification
\NLP~\eqref{eq:verifynlp} and the Pricing
Problem~\eqref{eq:pdwmodel:pricing}. By using more sophisticated heuristics for
the well-studied \cpp, we might be able to compute better optimality gaps for
instances where the positioning of circles into rectangles or rings is
difficult.
Surprisingly, even with a simple greedy heuristic our method works well as a
primal heuristic and finds for many instances better solutions than
\grasp~\cite{Pedroso2016}.

\section*{Acknowledgments}
This work has been supported by the Research Campus MODAL \emph{Mathematical
  Optimization and Data Analysis Laboratories} funded by the Federal Ministry of
Education and Research (BMBF Grant~05M14ZAM).  All responsibility for the content
of this publication is assumed by the authors.
 
\bibliographystyle{spmpsci}      
\bibliography{ring}

\newpage
\section*{Appendix}
\label{sec:appendix}

\begin{longtable}{@{\extracolsep{\fill}}lrrrrrrrrr}
  \caption{Detailed results for randomly generated instances.}
  \label{res:detailed}\\[-1ex]
  \tabledefline{\volcol}{volume-based dual bound}
  \tabledefline{\dualcol}{proven primal bound}
  \tabledefline{\primalcol}{proven dual bound}
  \tabledefline{\graspcol}{proven primal bound by \grasp heuristic}
  \tabledefline{\cpfoundcol}{total number of circular patterns found}
  \tabledefline{\cpcandcol}{total number of circular patterns candidates}
  \tabledefline{\rpcol}{total number of generated rectangular patterns}
  \tabledefline{\nodescol}{total number of branch-and-bound nodes}
  \tabledefline{\enumtimecol}{time to compute all circular patterns in enumeration experiment} [1.5ex] \toprule
  \textbf{instance} & \textbf{\volcol} & \textbf{\dualcol} & \textbf{\primalcol} & \textbf{\graspcol} & \textbf{\cpfoundcol} & \textbf{\cpcandcol} & \textbf{\rpcol} & \textbf{\nodescol} & \textbf{\enumtimecol}\\
  \midrule
  \endfirsthead
  \textbf{instance} & \textbf{\volcol} & \textbf{\dualcol} & \textbf{\primalcol} & \textbf{\graspcol} & \textbf{\cpfoundcol} & \textbf{\cpcandcol} & \textbf{\rpcol} & \textbf{\nodescol} & \textbf{\enumtimecol}\\
  \midrule
  \endhead
  i03\_2.0\_2.0\_05 &          8 &         17 &         17 &         17 &          4 &          0 &          6 &          1 &        0.4\\
i03\_2.0\_2.0\_10 &         22 &         34 &         38 &         38 &          5 &          0 &          7 &          3 &        1.4\\
i03\_2.0\_2.3\_05 &          7 &         16 &         16 &         16 &          4 &          0 &          6 &          1 &        0.5\\
i03\_2.0\_2.3\_10 &         23 &         33 &         45 &         49 &          4 &          0 &          9 &          3 &        0.9\\
i03\_2.0\_2.6\_05 &          8 &         23 &         23 &         23 &          3 &          0 &          7 &          1 &        0.6\\
i03\_2.0\_2.6\_10 &         14 &         41 &         41 &         41 &          4 &          0 &          9 &          1 &        0.3\\
i03\_2.0\_2.9\_05 &          8 &         16 &         16 &         16 &          4 &          0 &          9 &          1 &        0.4\\
i03\_2.0\_2.9\_10 &         14 &         18 &         27 &         27 &          4 &          0 &          8 &          3 &        0.7\\
\end{longtable}
 
\end{document}